\documentclass[11pt]{article}
\usepackage[T1]{fontenc}
\usepackage{lmodern,microtype,amsmath,amssymb,amsthm,mathtools}
\usepackage[margin=1in]{geometry}
\usepackage{booktabs,tikz,graphicx}
\usetikzlibrary{arrows.meta}
\definecolor{paperInk}{HTML}{233448}
\definecolor{paperMuted}{HTML}{637183}
\definecolor{paperTeal}{HTML}{007F86}
\definecolor{paperTealLight}{HTML}{E6F3F1}
\definecolor{paperGold}{HTML}{B87821}
\definecolor{paperGoldLight}{HTML}{FFF0D8}
\definecolor{paperCoral}{HTML}{C75249}
\definecolor{paperPurple}{HTML}{7564A2}
\definecolor{paperGrid}{HTML}{DFE5E9}
\definecolor{paperArrow}{HTML}{97A4AE}
\definecolor{paperDisc}{HTML}{F7DAD4}
\usepackage[colorlinks=true,linkcolor=blue,citecolor=blue,urlcolor=blue]{hyperref}
\hypersetup{pdftitle={Spectra of stochastic matrices with a Toeplitz interior},
  pdfauthor={Brecht Verbeken and Vincent Ginis}}
\newtheorem{theorem}{Theorem}[section]
\newtheorem{proposition}[theorem]{Proposition}
\newtheorem{lemma}[theorem]{Lemma}
\newtheorem{corollary}[theorem]{Corollary}
\theoremstyle{remark}

\newcommand{\Q}{\mathcal Q}
\newcommand{\Om}{\Omega}
\newcommand{\Sp}{\mathcal S}
\newcommand{\one}{\mathbf 1}
\newcommand{\conv}{\operatorname{conv}}
\newcommand{\spec}{\operatorname{spec}}
\newcommand{\rank}{\operatorname{rank}}
\newcommand{\tr}{\operatorname{tr}}
\newcommand{\diag}{\operatorname{diag}}
\newcommand{\ceil}[1]{\left\lceil#1\right\rceil}
\newcommand{\floor}[1]{\left\lfloor#1\right\rfloor}

\title{Spectra of stochastic matrices with a Toeplitz interior}
\author{%
  Brecht Verbeken\textsuperscript{1,2} and
  Vincent Ginis\textsuperscript{1,2,3}\\[0.8em]
  \parbox{0.96\textwidth}{\centering\small
    \textsuperscript{1}Department of Business Technology and Operations,\\
    Data Analytics Laboratory, Vrije Universiteit Brussel (VUB),\\
    Pleinlaan 2, 1050 Brussels, Belgium.\\[0.45em]
    \textsuperscript{2}imec-SMIT, Vrije Universiteit Brussel,\\
    Pleinlaan 9, 1050 Brussels, Belgium.\\[0.45em]
    \textsuperscript{3}School of Engineering and Applied Sciences,\\
    Harvard University, Cambridge, Massachusetts 02138, USA.\\[0.6em]
    \href{mailto:brecht.verbeken@vub.be}{brecht.verbeken@vub.be};
    \href{mailto:vincent.ginis@vub.be}{vincent.ginis@vub.be}
  }%
}
\date{}
\makeatletter
\renewcommand{\maketitle}{%
  \begin{center}
    {\LARGE\@title\par}\vspace{1em}
    {\large\@author\par}
  \end{center}\vspace{1em}}
\makeatother
\begin{document}
\maketitle

\begin{abstract}
For stochastic matrices of order $n$, the Karpelevi\v{c} region is the set of all possible eigenvalues. We study matrices whose first and last $m$ rows are arbitrary probability vectors, while the remaining rows have a Toeplitz pattern. As $m$ increases, their eigenvalue regions expand, first reaching the full Karpelevi\v{c} region at $m=\lceil(n-1)/4\rceil$ for $n\ge3$. Reproducing every complete stochastic spectrum, with multiplicities, requires $m=\lfloor n/2\rfloor$. For a boundary arc whose endpoints have orders $q<s$ as roots of unity, the least $m$ needed to attain the whole arc is $0$ for $q=1$ and $\left\lceil\frac12\left\lfloor n/q\right\rfloor\right\rceil$ for $q\ge2$. For a Karpelevi\v{c} boundary arc whose nonreal points between the endpoints first occur as stochastic eigenvalues at order $n\ge4$, fixing any one of these eigenvalues fixes the characteristic polynomial of every stochastic matrix of order $n$ with that eigenvalue. For these arcs and $m\ge1$, our family attains either the whole arc or only its endpoints. Further results describe the smaller eigenvalue regions and the angles at which their boundaries approach roots of unity.
\end{abstract}


\section{Introduction}
\label{sec:introduction}

For a fixed order $n$, the Karpelevi\v{c} region is the set of all possible eigenvalues of stochastic matrices. We ask how this region can be recovered from a structured family by allowing more rows to vary independently. We start with stochastic Toeplitz matrices, whose rows are cyclic shifts of one probability vector. These matrices describe a common transition rule on a fixed cyclic ordering of the states, and their eigenvalue region is the polygon spanned by the $n$th roots of unity. We enlarge this family by allowing the first and last $m$ rows to be arbitrary probability vectors, while the remaining rows retain the Toeplitz pattern. As $m$ increases, the attainable eigenvalue regions expand. We determine when they reach the full Karpelevi\v{c} region and when the family can reproduce every complete stochastic spectrum.

The Karpelevi\v{c} theorem describes which individual eigenvalues are possible
\cite{Karpelevic1951,Ito1997}. Asking which lists of eigenvalues can occur together is the
stochastic form of the nonnegative inverse eigenvalue problem
\cite{JohnsonMarijuanPaparellaPisonero2018}. We call the least values of $m$
needed for these two questions the saturation thresholds.

For a Markov chain, label the states $0,\ldots,n-1$ in a fixed cyclic order and allow
transition probabilities to vary independently on $2m$ consecutive states.
Cutting through the middle of this block puts $m$ free rows at each end of
the matrix. For every $r=0,\ldots,n-1$, the probability of moving from $i$ to $i+r\pmod n$ is the same at every interior state $i$. Arbitrary
permutations of the states need not preserve this pattern.

Throughout, stochastic means row-stochastic: $P\ge0$ and $P\one=\one$, where
$\one$ is the vector of ones. We use $J=\one\one^T/n$ for the matrix whose
entries are all $1/n$. For $v\in\mathbb C^n$, we write $\|v\|_\infty=\max_i|v_i|$. Fix an integer $0\le m\le\lfloor
n/2\rfloor$. We call the $n-2m$ consecutive rows indexed $m,\ldots,n-m-1$ the
interior rows. The other $2m$ rows, namely the first and last $m$ rows, are
the boundary rows. Their sets of row indices are
\[
 I_{n,m}=\{m,\ldots,n-m-1\},\qquad
 B_{n,m}=\{0,\ldots,n-1\}\setminus I_{n,m}.
\]
Row $i$ records the transition probabilities from state $i$ to all
$n$ states. We call a state an \emph{interior state} when its index
lies in $I_{n,m}$ and a \emph{boundary state} when its index lies
in $B_{n,m}$.
The class $\Q_{n,m}$ consists of stochastic matrices whose interior rows
satisfy the Toeplitz condition across all $n$ columns:
\begin{equation}
 P_{ij}=a_{j-i}\qquad(i\in I_{n,m},\ 0\le j<n).
 \label{eq:toeplitz}
\end{equation}
Here $a_k$ denotes the common entry on the diagonal $j-i=k$. The boundary rows are arbitrary probability vectors.
Proposition~\ref{prop:completion} shows that the interior rows coincide with
the corresponding rows of a stochastic circulant matrix. A circulant matrix
is generated from its first row by successive cyclic shifts.
Throughout, cyclic shifts of row vectors are to the right: a shift
by one position moves the last entry to the first position.
Successive interior rows differ by one such shift.
Figure~\ref{fig:cyclic-rows} shows these constraints and the cyclic ordering
for $n=10$ and $m=2$.

\begin{figure}[htbp]
\centering
\includegraphics[width=\textwidth]{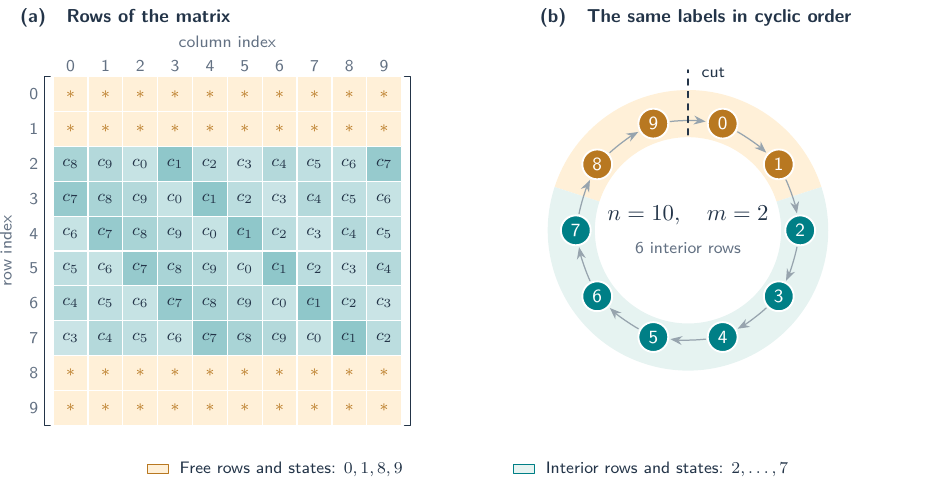}
\caption{The row constraints for $n=10$ and $m=2$. The six interior rows, indexed $2,\ldots,7$, are successive cyclic shifts of a probability vector $(c_0,\ldots,c_9)$. An asterisk denotes an arbitrary nonnegative entry, subject to each row summing to $1$. The free states $8,9,0,1$ form one consecutive block in the cyclic order. In the matrix ordering $0,\ldots,9$, they give the first two and last two rows. Arrows indicate the ordering of labels.}
\label{fig:cyclic-rows}
\end{figure}

Let $\Theta_n$ denote the Karpelevi\v{c} region. We write $\spec(P)$ for the
multiset of eigenvalues of $P$: repeated eigenvalues are counted with their
multiplicities, and their order is ignored. Let $\Sp_n$ be the set of these
multisets for all stochastic matrices of order $n$. For our class, define
\[
 \Om_{n,m}=\bigcup_{P\in\Q_{n,m}}\spec(P),\qquad
 \Sp_{n,m}=\{\spec(P):P\in\Q_{n,m}\}.
\]
The two saturation thresholds are the least values of $m$ needed to attain
every point of $\Theta_n$ and every multiset in $\Sp_n$, respectively:
\begin{equation}
 m_n^*=\min\{m:\Om_{n,m}=\Theta_n\},\qquad
 j_n^*=\min\{m:\Sp_{n,m}=\Sp_n\}.
 \label{eq:thresholds}
\end{equation}
The boundary of $\Theta_n$ consists of arcs joining roots of unity. If an
endpoint is written as $e^{2\pi i p/q}$ with $p/q$ in lowest terms, we call
$q$ its endpoint denominator; this is the order of that root of unity. For an
arc $A\subset\partial\Theta_n$, let $\kappa_n(A)$ be the least $m$ for which
every point of $A$, including both endpoints, is attained:
\[
 \kappa_n(A)=\min\{m:A\subseteq\Om_{n,m}\}.
\]

\begin{theorem}[Saturation thresholds]
\label{thm:main}
For $n=1,2$, one has $m_n^*=j_n^*=0$. For every $n\ge3$,
\begin{align}
 m_n^*&=\left\lceil\frac{n-1}{4}\right\rceil,
 \label{eq:individual-threshold}\\
 j_n^*&=\left\lfloor\frac n2\right\rfloor.
 \label{eq:joint-known}
\end{align}
If a Karpelevi\v{c} arc $A$ has endpoint denominators $q<s$, then
\begin{equation}
 \kappa_n(A)=
 \begin{cases}
 0,&q=1,\\[2pt]
 \left\lceil\dfrac12\left\lfloor n/q\right\rfloor\right\rceil,&q\ge2.
 \end{cases}
 \label{eq:arc-threshold}
\end{equation}
\end{theorem}

Thus the class attains the full Karpelevi\v{c} region while retaining the
Toeplitz pattern in about half of the rows. For $n\ge3$, attaining every
possible list of $n$ eigenvalues requires $m=\lfloor n/2\rfloor$, leaving one
interior row for odd $n$ and none for even $n$. In even order $n\ge4$,
retaining just two consecutive interior rows excludes some such lists.
The three formulas in Theorem~\ref{thm:main} are proved in
Sections~\ref{sec:chains}--\ref{sec:joint}.

Johnson and Paparella \cite{JohnsonPaparella2017}, and Kirkland, Laffey and \v{S}migoc \cite{KirklandLaffeySmigoc2020}, construct stochastic matrices from Ito polynomials. The polynomials and their types are introduced in Section~\ref{sec:chains}. Kirkland and \v{S}migoc \cite{KirklandSmigoc2022} classify Type~0 and Type~I realisations and the sparsest Type~II and Type~III realisations. The subsequent classifications of non-degenerate full-degree Type~II polynomials \cite{VerbekenGinisTypeII2026} and genuine Type~III polynomials over their nonzero parameter range \cite{VerbekenGinisTypeIII2026} concern all stochastic matrices with the specified characteristic polynomial. Munger, Nickerson and Paparella \cite{MungerNickersonPaparella2024} study how selected roots trace the boundary arcs. 

We construct the required matrices from directed paths of states. Transitions along each path are deterministic, while the transition probabilities at its last state are chosen to produce the required eigenvalues. These exceptional rows occupy consecutive positions and can therefore be placed among the free boundary rows.

For the lower bounds, we show that smaller values of $m$ leave out boundary eigenvalues. Let $\zeta$ be the endpoint of smaller order $q$ and put $d=\lfloor n/q\rfloor$. When $q$ does not divide $n$, eigenvalues approaching $\zeta$ along its arc would force a limiting matrix to have $\zeta$ as an eigenvalue of algebraic multiplicity at least $d$. On the other hand, the limiting matrix differs from a stochastic circulant matrix only in its $2m$ free rows. Since $\zeta$ is not an eigenvalue of the circulant matrix, a rank argument bounds that multiplicity by $2m$. The two bounds are incompatible when $2m<d$.

When the endpoint order divides $n$, we show that the permitted free rows do not enlarge the eigenvalue region near that endpoint, under the relevant bound on $m$. The region there remains the original polygon spanned by the $n$th roots of unity. Since the adjoining Karpelevi\v{c} arcs lie outside this polygon, their nearby points remain unattainable.

To complete the even-order argument, we construct eigenvalue lists that no stochastic matrix with two consecutive Toeplitz rows can reproduce. Order four is handled by an explicit matrix.

For orders $n=2q$ with odd $q\ge3$, we use the family shown in Figure~\ref{fig:paired-family}. Its states are arranged in $q$ pairs, with positive transition probabilities in both directions within each pair. One state from each pair belongs to a directed cycle. Other transitions between pairs can also occur, but only from one pair to the next in this cyclic order.

Let $P$ be a matrix from this family. Any stochastic matrix $Q$ with the same eigenvalues must satisfy
\[
 \operatorname{tr}(Q^k)=\operatorname{tr}(P^k)
\]
for every positive integer $k$. These traces sum the probabilities of returning to the starting state after $k$ steps, over all starting states. We use these equalities to show that, after relabelling its states, $Q$ must have the transition pattern shown in Figure~\ref{fig:paired-family}.

We choose the probabilities in $P$ so that no nonnegative matrix with the same first and second traces can have a smaller fourth trace. Since $Q$ has the same three traces, it also attains this minimum. We prove that this requires each state to have exactly one partner state with which transitions are possible in both directions. The two-by-two blocks for these pairs have known characteristic polynomials, so we know their contributions to the traces of other powers. For some powers, returns staying within individual pairs already account for the entire trace. Returns of those lengths visiting different pairs are therefore impossible. Another trace requires a return visiting different pairs to exist. Together, these restrictions force the cyclic arrangement of the pairs described above.

Within this transition pattern, two consecutive Toeplitz rows require additional relations between the transition probabilities. We show that matrices satisfying these relations produce a smaller-dimensional set of characteristic polynomials than the full Figure~\ref{fig:paired-family} family. Some eigenvalue lists from that family therefore cannot be reproduced with two consecutive Toeplitz rows.

To cover orders $8,12,16,\ldots$, we enlarge matrices from this family by placing a deterministic two-state cycle along the diagonal. This adds the eigenvalues $1$ and $-1$. We show that the resulting eigenvalue lists also cannot occur in a stochastic matrix with two consecutive Toeplitz rows.

This differs from the companion constructions for nonnegative Toeplitz matrices in \cite{CollaoSalasSoto2019} and the obstructions for permutative matrices in \cite{Loewy2016}: here only two consecutive rows retain the Toeplitz pattern.

In Section~\ref{sec:minimal-arcs}, we show that a nonreal eigenvalue strictly between the endpoints of a Karpelevi\v{c} boundary arc determines the entire characteristic polynomial, provided the matrix order $n\ge4$ is the smallest one in which that eigenvalue can occur. We use this to prove that, for such arcs and $m\ge1$, our family attains either the whole arc or only its endpoints.

Section~\ref{sec:regions} shows that, for $n\ge4$, each attainable region has a simple closed boundary whose distance from the origin varies continuously with direction. This strengthens the star-convexity studied in \cite{VagenendeVerbekenGuerry2025}. It also determines the directions from which attainable eigenvalues can approach roots of unity. 

Appendix~\ref{app:feasibility} studies how many positive entries are needed to realize a given eigenvalue, treating the free rows and the Toeplitz interior separately. Every attainable eigenvalue can be realized with at most three positive entries in each free boundary row. For a nonreal eigenvalue on the boundary of our attainable region, each free row needs at most two. For the Toeplitz interior, we bound the number of positive entries in the common probability vector whose cyclic shifts generate the interior rows. Every interior row then has that same number of positive entries. These results concern realizing one given eigenvalue; they do not require preserving the other eigenvalues.

Appendix~\ref{app:extensions} considers matrices in which two consecutive rows are related by a one-position cyclic shift, while the other rows are free probability vectors. Starting from such a matrix, we ask which small changes to its characteristic-polynomial coefficients can be obtained by small changes to its entries while keeping the row relation. When all entries of the starting matrix are strictly positive, every sufficiently small real coefficient change can be obtained in this way, provided the polynomial retains degree $n$, leading coefficient $1$, and the root $1$ required by stochasticity.

\begin{proposition}
\label{prop:completion}
If $n-2m\ge1$, every $P\in\Q_{n,m}$ agrees on rows $m,\ldots,n-m-1$
with a unique stochastic circulant matrix $C$. Conversely, replacing the
first $m$ and last $m$ rows of a stochastic circulant matrix by arbitrary
probability vectors produces a matrix in $\Q_{n,m}$. Thus
\begin{equation}
 P=C+E,\qquad E_{i,*}=0\ (i\in I_{n,m}),\qquad \rank E\leq2m,
 \label{eq:rank}
\end{equation}
where $E_{i,*}$ denotes row $i$ of $E$.
\end{proposition}
\begin{proof}
For consecutive interior rows, the Toeplitz identities give
$P_{i+1,j+1}=P_{ij}$ for $0\le j\le n-2$. Equality of their row sums then
gives $P_{i+1,0}=P_{i,n-1}$, so each row is obtained from the preceding
row by a cyclic shift one position to the right. Extending one interior row by cyclic shifts therefore gives the unique
stochastic circulant matrix $C$. Conversely, changing only boundary rows
preserves the interior identities. Since $E=P-C$ has at most $2m$ nonzero
rows, its rank is at most $2m$.
\end{proof}

If $n-2m=0$, there are no interior rows. We may then choose any stochastic
circulant matrix $C$ and put $E=P-C$; the decomposition and rank bound in
\eqref{eq:rank} still hold. In this case, $C$ is not uniquely determined by $P$.

For a probability vector $c=(c_0,\ldots,c_{n-1})$, write
\[
 C(c)_{ij}=c_{j-i\bmod n}\qquad(0\le i,j<n).
\]
In particular, $C_n$ is the permutation matrix with transition
probability $1$ from each state to its successor in the cyclic order:
$(C_n)_{i,i+1\bmod n}=1$.
Put $\omega=e^{2\pi i/n}$. The eigenvalues of $C(c)$ are
\[
 \lambda_r=\sum_{k=0}^{n-1}c_k\omega^{rk}\qquad(0\le r<n).
\]
Each $\lambda_r$ is a convex combination of $n$th roots of unity. Conversely, every
convex combination $\sum_k c_k\omega^k$ occurs as an eigenvalue of $C(c)$.
Together with Proposition~\ref{prop:completion}, this gives
\begin{equation}
 \Om_{n,0}=H_n:=\conv\{e^{2\pi i k/n}:0\le k<n\}.
 \label{eq:circulant-polygon}
\end{equation}
For $m=\floor{n/2}$, there is at most one interior row, so the Toeplitz
condition imposes no restriction and $\Q_{n,\floor{n/2}}$ contains every
stochastic matrix of order $n$. The classes $\Q_{n,m}$ are compact and
convex, and grow as $m$ increases. Their eigenvalue regions are compact.
For $n\ge2$ and $-1/(n-1)\le a\le1$, the matrix
$M_a=aI+(1-a)J$ is stochastic and circulant.
If $\spec(P)=\{1,\lambda_2,\ldots,\lambda_n\}$, the mixture
$tP+(1-t)M_a$ has eigenvalue $1$ and remaining eigenvalues
\begin{equation}
 \lambda_k(t)=t\lambda_k+(1-t)a
 \qquad(2\le k\le n,\ 0\le t\le1).
 \label{eq:averaging-spectrum}
\end{equation}
Indeed, use a basis whose first vector is $\one$. Since $P\one=\one$ and $J$ maps every vector into the span of $\one$, the matrices have block forms
\[
 P=\begin{pmatrix}1&*\\0&B\end{pmatrix},\qquad
 M_a=\begin{pmatrix}1&*\\0&aI_{n-1}\end{pmatrix}.
\]
Here $B$ has eigenvalues $\lambda_2,\ldots,\lambda_n$; the starred row blocks do not affect the eigenvalues. The mixture $tP+(1-t)M_a$ has lower-right block $tB+(1-t)aI_{n-1}$, whose eigenvalues are $t\lambda_k+(1-t)a$ for $2\le k\le n$. The mixture belongs to the same class by convexity.
The choices $a=0,1$ give mixing with $J$ and $I$, respectively.
Varying $a$ gives the stronger filling property
\begin{equation}
 \lambda\in\Om_{n,m}\quad\Longrightarrow\quad
 \conv\{-1/(n-1),1,\lambda\}\subseteq\Om_{n,m}.
 \label{eq:fill-from-one}
\end{equation}
For $\lambda=1$, the interval itself is attained by the matrices $M_a$.

\section{Boundary arcs and localised chain realisations}
\label{sec:chains}

The Karpelevi\v{c}--Ito theorem describes the boundary arcs using
neighbouring fractions in the Farey sequence of order $n$, the increasing
list of reduced fractions in $[0,1]$ with denominators at most $n$.
The endpoints of an arc can be written as $e^{2\pi ip/q}$ and
$e^{2\pi ip'/s}$, where the reduced fractions satisfy
\[
 q<s\le n,\qquad |ps-p'q|=1,\qquad q+s>n.
\]
For $q\ge2$, put $d=\floor{n/q}$ and $\beta=1-\alpha$, with
$0\le\alpha\le1$. The points on the arc are selected nonzero roots of
\begin{equation}
 \chi_{n;q,s,\alpha}(z)
 =z^{n-qd}(z^q-\beta)^d-\alpha^d z^{n-s}.
 \label{eq:ito}
\end{equation}
As $\alpha$ varies from $0$ to $1$, the theorem selects a continuous branch
of roots joining the two endpoints. Other roots of \eqref{eq:ito} need not
lie on the boundary. When $q=1$, the arcs are the two straight segments
from $1$ to $e^{2\pi i/n}$ and $e^{-2\pi i/n}$; we call these Type~0.

For $q\ge2$, put
\[
 N=\max(qd,s).
\]
For every nonreal point $\lambda$ between the endpoints, there exists
a stochastic matrix of order $N$ with eigenvalue $\lambda$, whereas no
stochastic matrix of order less than $N$ has that eigenvalue. Dividing \eqref{eq:ito} by $z^{n-N}$ removes
the additional zero roots and gives the reduced Ito polynomial of degree
$N$. We distinguish three cases:
\[
\begin{array}{c|c|c}
 \text{type}&\text{condition}&\text{least order for nonreal points }N\\ \hline
 \mathrm{I}&d=1&s\\
 \mathrm{II}&d\ge2,\ qd>s&qd\\
 \mathrm{III}&d\ge2,\ s>qd&s.
\end{array}
\]
Since $q$ and $s$ are coprime and $q\ge2$, equality $s=qd$ is impossible.
The order $n$ can exceed $N$; the same arc then appears on the boundaries
of both $\Theta_N$ and $\Theta_n$. The real segment $[-1,-1/2]$ on
the order-three boundary is exceptional: all its points already occur
as eigenvalues at order $2$.

The next construction allows the exponent $d$ in \eqref{eq:ito} to be
replaced by a positive integer $r$. It also allows the matrix order $n$
to increase with $q,r,s$ and $\alpha$ fixed. Each added state then
contributes an additional zero eigenvalue, while all previous eigenvalues
are retained.

\begin{theorem}
\label{thm:chains}
Let $q,r,s$ be integers with $q\ge2$, $r\ge1$, $qr\le n$ and
$r\le s\le n$. For every $\alpha\in[0,1]$, there exists a stochastic
matrix $P$ of order $n$ with characteristic polynomial
\begin{equation}
 z^{n-qr}(z^q-1+\alpha)^r-\alpha^r z^{n-s}.
 \label{eq:general-chain-poly}
\end{equation}
In every row outside a block of $r$ consecutive rows, the transition is
from state $i$ to state $i+r\pmod n$ with probability one. One can add the same
integer to every state label, modulo $n$, so that the matrix belongs to
$\Q_{n,\ceil{r/2}}$. It has at most $n+r$ strictly positive entries.
\end{theorem}
\begin{proof}
Partition the states $0,\ldots,n-1$ into $r$ chains
\[
 H_j=(j,j+r,j+2r,\ldots)\qquad(0\le j<r),
\]
stopping each chain before its labels reach $n$. Write $h_j$ for the
number of states in $H_j$ and $e_j$ for its last state. Since $qr\le n$,
every chain contains at least $q$ states. The last states of the chains
are exactly $n-r,\ldots,n-1$. Choose integers
\[
 1\le\ell_j\le h_j,\qquad\sum_{j=0}^{r-1}\ell_j=s.
\]
Such a choice exists: starting with every $\ell_j=1$, we can increase the
sum one at a time up to $\sum_j h_j=n$, and $r\le s\le n$.

For each state other than the last state of a chain, the transition
probability to the next state in that chain is $1$. From $e_j$, the
transition probabilities are $\beta=1-\alpha$ to the first of the last
$q$ states of $H_j$, and $\alpha$ to the first of the last
$\ell_{j+1}$ states of $H_{j+1}$. Read chain indices modulo $r$. If both rules lead from $e_j$ to the same state $k$, set $P_{e_j,k}=\alpha+\beta=1$. The outgoing probabilities sum to one in every
row, so this defines a stochastic matrix $P$.

For $z\ne0$, put $x_j=v_{e_j}$. The equations $Pv=zv$ along a chain give
$v_{e_j-kr}=z^{-k}x_j$ whenever that state belongs to $H_j$. The equations
at the last states therefore become
\[
 (z^q-\beta)x_j=\alpha z^{q-\ell_{j+1}}x_{j+1}.
\]
To compute the characteristic polynomial with multiplicities, perform the
same elimination on $zI-P$. The $n-r$ eliminated coordinates contribute
the determinant factor $z^{n-r}$. The remaining matrix has diagonal
contributions $z-\beta z^{1-q}$ and contributions
$-\alpha z^{1-\ell_{j+1}}$ in positions $(j,j+1)$, with indices modulo
$r$. Its determinant is
\[
 z^{r(1-q)}(z^q-\beta)^r-\alpha^r z^{r-s}.
\]
When $r=1$, there is only one chain, so $\ell_0=s$. The reduced matrix has a single row and column:
\[
 \begin{pmatrix}z-\beta z^{1-q}-\alpha z^{1-s}\end{pmatrix}.
\]
Its determinant is this entry itself, which is exactly the preceding formula with $r=1$. Multiplying by $z^{n-r}$ gives
\eqref{eq:general-chain-poly}. The calculation assumes $z\ne0$. Both sides of the resulting identity are polynomials and therefore continuous at zero. Taking $z\to0$ proves the identity also at $z=0$.

Only the rows indexed by $e_j$ can differ from the transition
$i\mapsto i+r\pmod n$. Put $m=\ceil{r/2}$ and add $r-m$ to every label,
modulo $n$. The last states of the chains are then split between the last
$m$ labels and the first $r-m$ labels. Since $r-m\le m$, all these rows
lie among the first $m$ or last $m$ rows. Adding the same integer to all state labels leaves the rule
$i\mapsto i+r\pmod n$ unchanged, so the other rows still follow the same
circulant shift. The relabelled matrix therefore belongs to $\Q_{n,m}$. Finally,
each of the $n-r$ other rows has one strictly positive entry, and each of the $r$
last-state rows has at most two, giving at most $(n-r)+2r=n+r$ strictly positive
entries.
\end{proof}

Setting $r=\floor{n/q}$ in Theorem~\ref{thm:chains} gives a matrix with
characteristic polynomial \eqref{eq:ito}. The Type~0 segments are obtained
from $(1-\alpha)I+\alpha C_n^{\pm1}$. Consequently,
\begin{equation}
 \kappa_n(A)\le\ceil{\frac12\floor{n/q}}\quad(q\ge2),\qquad
 m_n^*\le\ceil{\frac{n-1}{4}}.
 \label{eq:upper}
\end{equation}
The largest value of the first bound occurs at $q=2$ and equals
$\ceil{(n-1)/4}$. To obtain the second bound for the whole region, take
$z\in\Theta_n\setminus\{1\}$ and let $w$ be the point of $\Theta_n$
farthest from $1$ on the ray from $1$ through $z$. Compactness guarantees
that $w$ exists, and it lies on $\partial\Theta_n$: otherwise the ray
could be extended slightly farther within $\Theta_n$. Since $z$ lies on
the segment from $1$ to $w$, property~\eqref{eq:fill-from-one} places $z$
in the same region $\Om_{n,m}$ as $w$.

\section{Invariant polygons and peripheral tangents}
\label{sec:polygons}

Let $P$ be a stochastic matrix and let $v\ne0$ satisfy $Pv=\lambda v$. For each
$i$, the identity $\lambda v_i=\sum_{j=0}^{n-1}P_{ij}v_j$ expresses
$\lambda v_i$ as a convex combination of the eigenvector coordinates.
Consequently,
\begin{equation}
 \lambda K(v)\subseteq K(v),\qquad
 K(v)=\conv\{v_0,\ldots,v_{n-1}\}.
 \label{eq:polygon-invariant}
\end{equation}
Repeated application gives $\lambda^kK(v)\subseteq K(v)$ for every integer
$k\ge0$. The next lemma relates the number of vertices to the rotations
and contractions that keep a polygon inside itself.

\begin{lemma}
\label{lem:sector}
Let $K\subset\mathbb C$ be a convex polygon with $N\ge3$ vertices and
$0$ in its interior. Fix $0<\theta<\pi/2$ and suppose that
\[
 e^{-t e^{i\theta}}K\subseteq K\qquad(t\ge0).
\]
This map multiplies distances from $0$ by $e^{-t\cos\theta}$ and rotates
points clockwise about $0$ through angle $t\sin\theta$. Then
$\theta\le\pi/2-\pi/N$. If equality holds, $K$ is a regular $N$-gon
centred at $0$.
\end{lemma}
\begin{proof}
List the vertices counterclockwise as $z_i=R_i e^{i\phi_i}$, with
$R_i>0$, and read the indices cyclically. Let $\delta_i$ be the
counterclockwise angle from the ray through $z_{i-1}$ to the ray through
$z_i$. Since $0$ lies inside $K$,
\[
 0<\delta_i<\pi,\qquad\sum_i\delta_i=2\pi.
\]
Let $\gamma_i$ be the angle of the outward unit normal $\nu_i=e^{i\gamma_i}$ to the edge from $z_{i-1}$ to $z_i$, and put $a_i=\phi_i-\gamma_i$. The line containing this edge has equation $\operatorname{Re}(\overline{\nu_i}z)=h$, and $K$ lies in the half-plane $\operatorname{Re}(\overline{\nu_i}z)\le h$. Since $0$ is in the interior of $K$, we have $h>0$. Both endpoints lie on this line, so
\[
 R_i\cos a_i=R_{i-1}\cos(a_i-\delta_i)=h>0.
\]
Choosing the angle representatives accordingly gives
\[
 -\frac\pi2<a_i-\delta_i<a_i<\frac\pi2,\qquad
 \frac{R_i}{R_{i-1}}
 =\frac{\cos(a_i-\delta_i)}{\cos a_i}
 =\cos\delta_i+\tan a_i\sin\delta_i.
\]

Put $p_i(t)=e^{-t e^{i\theta}}z_i$. The assumption gives $p_i(t)\in K$ for $t\ge0$, and hence
\[
 \operatorname{Re}(\overline{\nu_i}p_i(t))
 \le h=\operatorname{Re}(\overline{\nu_i}p_i(0)).
\]
Taking the right derivative at $t=0$ gives
\[
 \operatorname{Re}(\overline{\nu_i}p_i'(0))\le0.
\]
Since $p_i'(0)=-e^{i\theta}z_i$, this is $-R_i\cos(a_i+\theta)\le0$. With the above range for $a_i$, this gives $a_i\le\pi/2-\theta$. Since $\tan$ is increasing
on $(-\pi/2,\pi/2)$, the ratio above satisfies
\begin{equation}
 0<\frac{R_i}{R_{i-1}}
 \le\frac{\sin(\theta+\delta_i)}{\sin\theta}.
 \label{eq:sector-ratio}
\end{equation}
In particular, $\theta+\delta_i\in(0,\pi)$. The ratios on the left
multiply to $1$. Taking products and using concavity of $\log\sin$
on $(0,\pi)$ gives
\[
 1\le\prod_i\frac{\sin(\theta+\delta_i)}{\sin\theta}
 \le\left[\frac{\sin(\theta+2\pi/N)}{\sin\theta}\right]^N.
\]
Hence $\sin(\theta+2\pi/N)\ge\sin\theta$. Since
$\theta<\theta+2\pi/N<\pi$, it follows that
$\theta+2\pi/N\le\pi-\theta$, or $\theta\le\pi/2-\pi/N$.

If equality holds, strict concavity of $\log\sin$ forces all
$\delta_i$ to equal $2\pi/N$. The right-hand side of
\eqref{eq:sector-ratio} is then $1$ for every $i$. Because the positive
ratios $R_i/R_{i-1}$ have product $1$ and are each at most $1$, they
must all equal $1$. This implies that all vertices have the same distance from $0$
and are equally spaced in angle, so $K$ is a regular $N$-gon centred
at $0$.
\end{proof}

\begin{theorem}
\label{thm:tangent}
Let $P_j$, $j=1,2,\ldots$, be stochastic matrices of order $n$ whose
entries converge to the corresponding entries of a stochastic matrix
$P_0$. We denote this convergence by $P_j\to P_0$. Let
$\lambda_j\in\spec(P_j)$ converge to a primitive $q$th root of unity
$\zeta$, with $q\ge2$. Put $w_j=1-\overline\zeta\lambda_j$ and assume
$w_j\ne0$ and $\arg w_j\to\theta$. If $\zeta$ has algebraic multiplicity
$r$ as an eigenvalue of $P_0$, then
\begin{equation}
 |\theta|\le\frac\pi2-\frac\pi{qr}.
 \label{eq:tangent-bound}
\end{equation}
If $n=qd$ and $|\theta|=\pi/2-\pi/n>0$, the states of $P_0$ split into
$d$ disjoint cycles of length $q$, with every transition around each cycle
having probability $1$. In this case, for any eigenvectors satisfying
$P_jv_j=\lambda_jv_j$ and $\|v_j\|_\infty=1$, every convergent subsequence
has a limit whose coordinates are the $n$ distinct vertices of a regular
$n$-gon centred at $0$.
\end{theorem}
\begin{proof}
Since $|\lambda_j|\le1$ and $w_j\ne0$, we have $\Re w_j>0$, so
$\arg w_j\in(-\pi/2,\pi/2)$. Choose eigenvectors with
$\|v_j\|_\infty=1$ and pass to a subsequence for which $v_j\to v$.
Then $\|v\|_\infty=1$ and $P_0v=\zeta v$.

For fixed $t\ge0$, put $k_j=\floor{t/(q|w_j|)}$. Since
$\lambda_j=\zeta(1-w_j)$ and $\zeta^q=1$,
\[
 \lambda_j^{qk_j}=(1-w_j)^{qk_j}\longrightarrow e^{-t e^{i\theta}}.
\]
Apply \eqref{eq:polygon-invariant} repeatedly to obtain
$\lambda_j^{qk_j}K(v_j)\subseteq K(v_j)$, and pass to the limit:
\begin{equation}
 e^{-t e^{i\theta}}K(v)\subseteq K(v)\qquad(t\ge0).
 \label{eq:limit-spiral}
\end{equation}
If $\theta=\pm\pi/2$, this would make $K(v)$ invariant under every
rotation, which is impossible for a nonzero polygon. If $\theta=0$,
the required bound is immediate. Taking complex conjugates reduces the
remaining cases to $0<\theta<\pi/2$. By \eqref{eq:limit-spiral}, the
orbit of any nonzero point of $K(v)$ stays inside $K(v)$ and passes
through every direction. Choose three points of this orbit at angles
separated by $2\pi/3$. Their triangle contains $0$ in its interior and
lies in $K(v)$, so $0$ is also an interior point of $K(v)$.

We next bound the number of vertices of $K(v)$. Every eigenvalue of
$P_0$ on the unit circle is semisimple: a larger Jordan block would make
the powers of $P_0$ unbounded, whereas every power is stochastic. A closed
irreducible class is a group of states in which each state can reach every
other and from which the chain cannot leave. Such a class contributes
$\zeta$ exactly when its period is divisible by $q$, and then contributes
it with multiplicity $1$. Thus there are
exactly $r$ such classes. On each of them, the coordinates of $v$ take
at most the $q$ values $a,\zeta a,\ldots,\zeta^{q-1}a$ for some $a$.
On all other closed classes, the coordinates are $0$.

Because $P_0^qv=v$, each coordinate $v_i$ corresponding to a state $i$
outside these closed classes is a convex combination of the coordinates
$v_k$ for states $k$ inside them. The coefficient of $v_k$ is the probability
that, starting from $i$, the chain with transition matrix $P_0^q$ first
enters the union of the closed classes at state $k$. These probabilities
sum to $1$.
Also, $0$ lies in the convex hull of each set
$\{a,\zeta a,\ldots,\zeta^{q-1}a\}$, because its elements sum to $0$.
Consequently, $K(v)$ has at most $qr$ vertices. Lemma~\ref{lem:sector}
now gives \eqref{eq:tangent-bound}.

If $n=qd$ and we have equality $|\theta|=\pi/2-\pi/n>0$, inequality~\eqref{eq:tangent-bound} gives $qr\ge n$. Each of the
$r$ contributing classes has at least $q$ states, so $qr\le n$.
Hence $r=d$, every state belongs to one of these classes, and each class
has exactly $q$ states. Such a class has period $q$ and must therefore
be a cycle whose transitions all have probability $1$. Let $N$ be the number of vertices of $K(v)$. Lemma~\ref{lem:sector} gives $|\theta|\le\pi/2-\pi/N$. Since $|\theta|=\pi/2-\pi/n$, this implies $N\ge n$. The polygon is the convex hull of $n$ coordinates, so $N\le n$ and hence $N=n$. Lemma~\ref{lem:sector} then shows that the polygon is regular, so the $n$ coordinates of $v$ are its distinct vertices.
\end{proof}

Consider the arc followed clockwise from $\zeta=e^{2\pi ip/q}$ towards
$e^{2\pi ip'/s}$, so that $ps-p'q=1$. Put $d=\floor{n/q}$ and
$u=e^{-2\pi i/(qd)}$. As $\alpha\downarrow0$, the eigenvalue on this arc
satisfies
\begin{equation}
 \lambda(\alpha)=\zeta\left(1+\frac{u-1}{q}\alpha+O(\alpha^2)\right),
 \qquad
 \arg(1-\overline\zeta\lambda(\alpha))\longrightarrow
 \frac\pi2-\frac\pi{qd}.
 \label{eq:endpoint-expansion}
\end{equation}
To obtain the expansion, write \eqref{eq:ito} for a nonzero root as
\[
 (z^q-1+\alpha)^d=\alpha^d z^{qd-s}.
\]
Substitute $z=\zeta(1+\alpha a)$, divide by $\alpha^d$, and let
$\alpha\to0$. Since $\zeta^{qd-s}=e^{-2\pi i/q}$, this gives
\[
 (1+qa)^d=e^{-2\pi i/q}.
\]
The clockwise boundary branch selects $1+qa=u$. This is a simple root
of the equation in $a$, so the implicit function theorem gives
$a=a(\alpha)$ with $a(0)=(u-1)/q$, yielding the first formula.
Finally,
\[
 1-u=2i e^{-\pi i/(qd)}\sin(\pi/(qd)),
\]
so $\arg(1-u)=\pi/2-\pi/(qd)$, which gives the stated angle limit.

\begin{proposition}
\label{prop:nondivisor}
Let $A$ be a Karpelevi\v{c} boundary arc with endpoint denominators $q<s$,
and let $\zeta$ be its endpoint of order $q$. If $q$ does not divide $n$
and $2m<\floor{n/q}$, then all points of $A$ sufficiently close to
$\zeta$, apart from $\zeta$ itself, lie outside $\Om_{n,m}$.
\end{proposition}
\begin{proof}
Suppose instead that there are matrices $P_j\in\Q_{n,m}$ with eigenvalues
on the arc at positive parameters $\alpha_j$ tending to $0$. Compactness
of $\Q_{n,m}$ gives a subsequence whose entries converge to those of a
matrix $P_0\in\Q_{n,m}$. Put $d=\floor{n/q}$, and let $r$ be the
algebraic multiplicity of $\zeta$ in $P_0$. Comparing
\eqref{eq:tangent-bound} and \eqref{eq:endpoint-expansion}, taking complex
conjugates if necessary, gives $r\ge d$.

Write $P_0=C+E$ as in \eqref{eq:rank}, with $C$ stochastic and circulant
and $\rank E\le2m$. Each eigenvalue of $C$ is a convex combination of
$n$th roots of unity. Such a combination can have modulus $1$ only if
all roots with positive weight are equal. Hence every eigenvalue of $C$
on the unit circle is itself an $n$th root of unity. Because $\zeta$ is
a primitive $q$th root and $q$ does not divide $n$, it is not an $n$th
root. Thus $C-\zeta I$ is invertible.

Set $T=-(C-\zeta I)^{-1}E$. Every vector $x$ satisfying $P_0x=\zeta x$ obeys $x=Tx$. Thus every $\zeta$-eigenvector is an output of $T$, so the eigenspace is contained in $\operatorname{im}T$. Since $C-\zeta I$ is invertible,
\[
 \dim\operatorname{im}T=\rank T=\rank E\le2m.
\]
The eigenspace therefore has dimension at most $2m$. Because
eigenvalues on the unit circle are semisimple, this dimension equals the
algebraic multiplicity $r$. This contradicts $r\ge d>2m$.
\end{proof}

\section{Divisor roots and local polygon boundaries}
\label{sec:divisor}

When $q$ divides $n$, the endpoint $\zeta$ is an $n$th root of unity
and can be an eigenvalue of a stochastic circulant matrix. In the decomposition
$P=C+E$, the matrix $C-\zeta I$ may then be singular, so the preceding
argument does not apply. We will show that, if $2m<n/q$, the region $\Om_{n,m}$ agrees
with the polygon $H_n$ in a sufficiently small neighbourhood of $\zeta$.

\begin{theorem}[Local equality with the circulant polygon]
\label{thm:divisor}
\label{thm:local-polygon}
Suppose $n=qd$, $q\ge2$, $d\ge3$ and $2m<d$. For every primitive
$q$th root of unity $\zeta$, there is an $\varepsilon>0$ such that,
with $U=\{z\in\mathbb C:|z-\zeta|<\varepsilon\}$,
\begin{equation}
 \Om_{n,m}\cap U=H_n\cap U.
 \label{eq:local-polygon}
\end{equation}
On each of the two Karpelevi\v{c} boundary arcs meeting $\zeta$, all
points sufficiently close to $\zeta$ but not equal to it therefore
lie outside $\Om_{n,m}$.
\end{theorem}

The radius $\varepsilon$ and the constants in the proof may depend on
$n,m,\zeta$. The result gives no radius that works uniformly as $n$ varies.
Figure~\ref{fig:local-polygon-comparison} illustrates the two boundaries
for $n=9$ and $m=1$.

\begin{figure}[htbp]
\centering
\includegraphics[width=\textwidth]{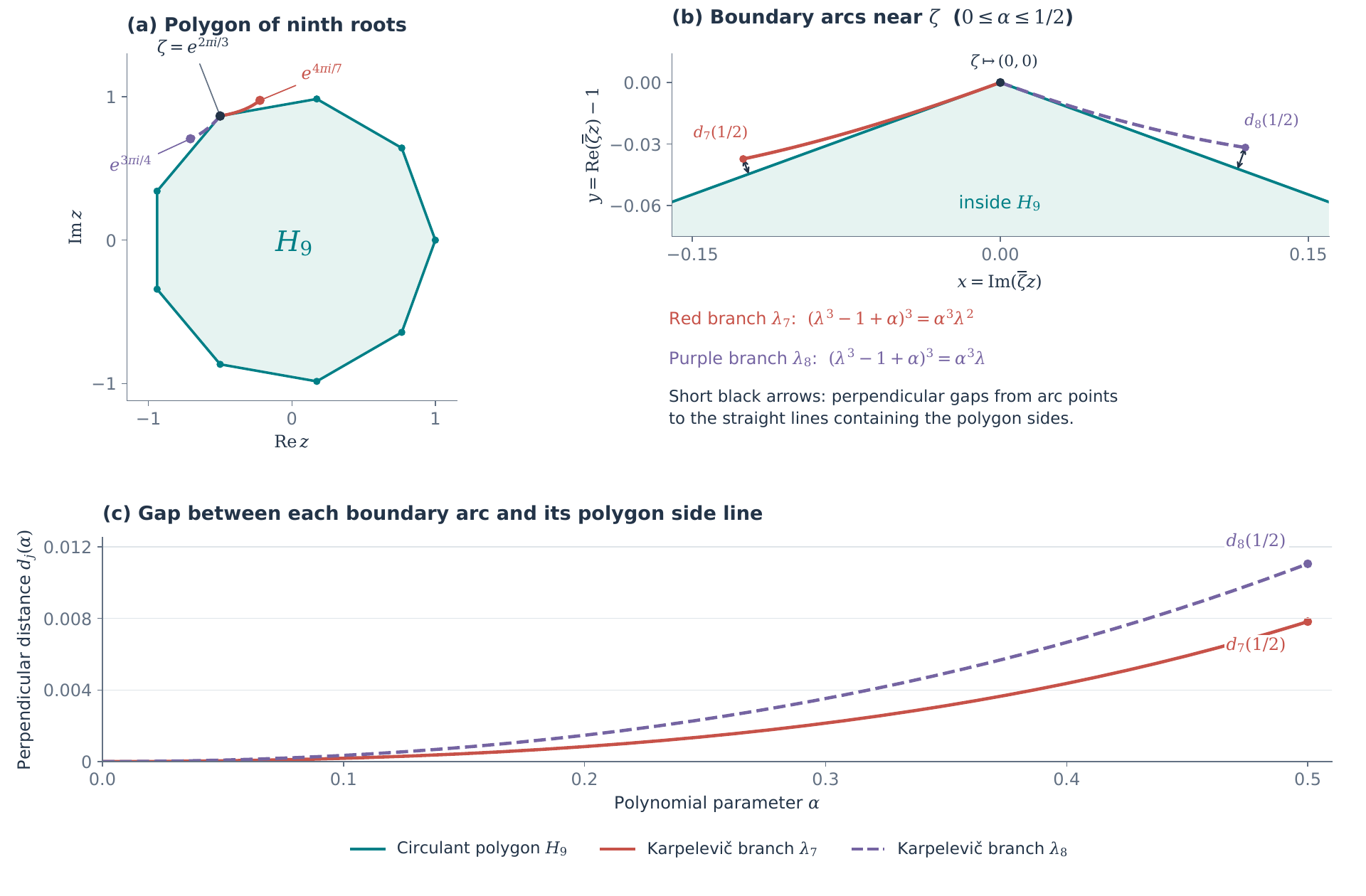}
\caption{The circulant polygon $H_9$ and two Karpelevi\v{c} boundary arcs meeting $\zeta=e^{2\pi i/3}$. (a) $H_9$ is the regular nonagon whose vertices are the ninth roots of unity; it is the eigenvalue region of stochastic circulant matrices of order $9$. The red branch $\lambda_7(\alpha)$ joins $\zeta$ to $e^{4\pi i/7}$, and the purple branch $\lambda_8(\alpha)$ joins it to $e^{3\pi i/4}$ as $\alpha$ runs from $0$ to $1$. They are the selected polynomial roots specified in panel (b). (b) A zoom for $0\le\alpha\le1/2$, in coordinates $x=\operatorname{Im}(\overline\zeta z)$ and $y=\operatorname{Re}(\overline\zeta z)-1$, which move $\zeta$ to $(0,0)$. The shaded part lies inside $H_9$. The black arrows show the shortest distances from the arc points at $\alpha=1/2$ to the straight lines extending the adjacent polygon sides. (c) These distances are plotted as $d_j(\alpha)=\operatorname{dist}(\lambda_j(\alpha),L_j)$, where $L_j$ is the corresponding side line. Panel (b) shows the arc points on the outer side of these lines; panel (c) quantifies their separation despite the shared tangent directions at $\zeta$. By Theorem~\ref{thm:local-polygon}, $\Omega_{9,1}$ agrees with $H_9$ sufficiently close to $\zeta$, so nearby points on these arcs cannot be eigenvalues of matrices in our family.}
\label{fig:local-polygon-comparison}
\end{figure}

For $m=0$, the equality follows from $\Om_{n,0}=H_n$. Therefore suppose $m\ge1$
and the local equality fails. Choose eigenvalues
$\lambda_j\in\Om_{n,m}\setminus H_n$ tending to $\zeta$ and matrices
$P_j\in\Q_{n,m}$ with these eigenvalues. Then pass to subsequences so that the
entries of $P_j$ converge to those of $P_0$, the corresponding eigenvectors
$v_j$, normalised by $\|v_j\|_\infty=1$, converge to $v$, and the angles
$\arg(1-\overline\zeta\lambda_j)$ converge.

Near $\zeta$, the polygon is the intersection of the two half-planes
bounded by the lines joining $\zeta$ to its neighbouring vertices. Select a subsequence that violates
the same one of these two half-planes. Theorem~\ref{thm:tangent} bounds the absolute value of the limit of
$\arg(1-\overline\zeta\lambda_j)$ by $\pi/2-\pi/n$, because $q$ times the multiplicity of $\zeta$ cannot
exceed $n$. These are precisely the two edge directions. Since the points lie outside the polygon, they must approach
along one of the two edges. Taking complex conjugates if necessary,
choose the clockwise edge and put $u=e^{-2\pi i/n}$. Define
\[
 \alpha_j=q\operatorname{Re}
 \frac{\lambda_j-\zeta}{\zeta(u-1)}.
\]
Then $\alpha_j>0$ for sufficiently large $j$, and
\[
 \lambda_j=\zeta\left(1+\frac{u-1}{q}\alpha_j+o(\alpha_j)\right).
\]
The limit of $\arg(1-\overline\zeta\lambda_j)$ therefore has
absolute value $\pi/2-\pi/n$.
Theorem~\ref{thm:tangent} therefore implies that $P_0$ consists of $d$ disjoint cycles of length $q$,
with probability $1$ for every transition around each cycle, and the
coordinates of $v$ are the distinct vertices of a regular $n$-gon
centred at $0$. Its radius is $1$ because $\|v\|_\infty=1$.

Every interior row $i$ of $P_0$ assigns probability $1$ to
$i+h\pmod n$, with the same integer $h$ for all interior rows. The $2m<d$ states with labels in $B_{n,m}$ meet fewer
than $d$ cycles, so at least one complete cycle has all its labels in
$I_{n,m}$. Following that cycle shows that adding $h$ modulo $n$ has
order $q$. Thus
\[
 \gcd(h,n)=d,\qquad R=C_n^h.
\]
The cycles of $R$ correspond to the residue classes modulo $d$. Because $B_{n,m}$
is a block of $2m<d$ consecutive labels in cyclic order, it meets each
such cycle at most once. If a cycle of $R$ contains a boundary state $b\in B_{n,m}$,
its other rows already form the path from $R(b)$ to $b$. The remaining
transition must close this path: linking two or more such paths would
produce a cycle of length at least $2q$, whereas every cycle of $P_0$
has length $q$. Hence $P_0=R$ on every row.

Let $K$ send each state to the state whose coordinate is its clockwise
neighbour in the polygon. Thus $Kv=uv$, $K$ is an $n$-cycle, and
$K$ commutes with $R$. Put $S=RK$ and
\[
 \mathcal A=\frac1q\sum_{k=0}^{q-1}R^k.
\]
This is the orthogonal projection onto vectors constant on each
$R$-cycle. For every real mean-zero vector $f$ constant on $I_{n,m}$,
\begin{equation}
 \|\mathcal Af\|^2\ge\kappa\|f\|^2,\qquad
 \kappa=\frac{d-2m}{n-2m}>0.
 \label{eq:cycle-averaging}
\end{equation}
Indeed, write $f=c\one+x$, where $x$ is supported on $B_{n,m}$.
The mean-zero condition gives $nc=-\sum_i x_i$. Since the boundary
indices lie on distinct $R$-cycles,
\[
 \|f\|^2=\|x\|^2-\frac{(\sum_i x_i)^2}{n},\qquad
 \|\mathcal Af\|^2=\frac{\|x\|^2}{q}-\frac{(\sum_i x_i)^2}{n}.
\]
The bound $(\sum_i x_i)^2\le2m\|x\|^2$ proves
\eqref{eq:cycle-averaging}.

\begin{lemma}
\label{lem:first-order}
Along the chosen subsequence,
\begin{equation}
 P_j=R+\frac{\alpha_j}{q}(S-R)+o(\alpha_j).
 \label{eq:first-order}
\end{equation}
\end{lemma}
\begin{proof}
First put $\eta_j=\max_i(1-(P_j)_{i,R(i)})$, so
$P_j-R=O(\eta_j)$. The fixed real linear functionals
$\phi_i(z)=\operatorname{Re}(\overline v_i z)$ expose the individual
vertices of the limiting regular polygon. For large $j$,
$\phi_i(v_{j,i}-v_{j,k})\ge c>0$ whenever $k\ne i$.
Applying $\phi_i$ to $P_j^qv_j=\lambda_j^qv_j$ gives
\[
 c\sum_{k\ne i}(P_j^q)_{ik}
 \le\phi_i((1-\lambda_j^q)v_{j,i})=O(\alpha_j).
\]
Choose a state $i$ whose row attains $\eta_j$. Start at $i$ and first move to a state $k\ne R(i)$. Following $R$ for the remaining $q-1$ steps ends at $R^{q-1}(k)$, which differs from $i$: since $R^q=I$, returning to $i$ would require $k=R(i)$. These paths have total probability at least
$\eta_j(1-\eta_j)^{q-1}$. Since $\eta_j\to0$, this proves
$P_j-R=O(\alpha_j)$.

Every subsequence now has a further subsequence on which
$(P_j-R)/\alpha_j$ converges to a matrix $A$. Expanding the $q$th
power gives
\[
 G=\sum_{k=0}^{q-1}R^kAR^{q-1-k},\qquad
 Gv=(u-1)v=(K-I)v.
\]
The off-diagonal entries of $G$ are nonnegative and its row sums
are zero. At vertex $v_i$, the chord $v_{K(i)}-v_i$ lies on the
clockwise polygon edge; all other chords point strictly inward
from its line. A nonnegative combination can give this chord only
by using that neighbour, with coefficient $1$. Hence $G=K-I$.

The matrix $H=AR^{-1}$ also has nonnegative off-diagonal entries
and zero row sums, and
\[
 \sum_{k=0}^{q-1}R^kHR^{-k}=K-I.
\]
Its off-diagonal entries cannot cancel, so
\[
 A=\diag(a)(S-R),\qquad a_i\ge0,\qquad
 \sum_{k=0}^{q-1}a_{R^k(i)}=1.
\]
The interior entries $A_{i,R(i)}=-a_i$ are equal because the
interior rows are cyclic shifts. Thus $a-q^{-1}\one$ is constant
on $I_{n,m}$, has mean zero, and has zero cycle average.
Equation~\eqref{eq:cycle-averaging} forces $a=q^{-1}\one$.
Every subsequential limit is therefore $(S-R)/q$.
\end{proof}

\begin{proof}[Completion of the proof of Theorem~\ref{thm:divisor}]
The matrix $S$ is also circulant. Indeed, \eqref{eq:first-order} shows
that $S(i)=i+h'\pmod n$ for one integer $h'$ on every interior row.
Commutation with $R$ extends this formula around each $R$-cycle,
which has at most one boundary state. Hence $K=R^{-1}S$ is a
circulant $n$-cycle. Rotate $v$ so that $v_0=1$; its coordinates
are then the distinct $n$th roots of unity.

Put $t_j=n^{-1}\sum_i(P_j)_{i,S(i)}$ and omit the subscript $j$.
Write
\begin{equation}
 P=R+t\diag(\one+F)(S-R)+tG,\qquad
 t\sim\alpha_j/q,\quad F\to0,\quad G\to0.
 \label{eq:local-decomposition}
\end{equation}
Here $F_i=P_{i,S(i)}/t-1$ is real, has mean zero, and is constant
on $I_{n,m}$. Outside the $R,S$ positions, $G_{ik}=P_{ik}/t\ge0$.
Its $S$ entries are zero, and its $R$ entry in each row is minus
the sum of that row's outside entries. Thus, with
\[
 \ell(G)=\sum_i\sum_{k\notin\{R(i),S(i)\}}G_{ik},
 \qquad \|G\|\le C\ell(G),
\]
the total outside mass controls $G$. All norms in this calculation
are Euclidean or operator norms, and all constants may depend on
the fixed limiting configuration.

Let $V=\diag(v)$, $\widehat G=\zeta^{-1}V^*GV$ and
$\lambda=\zeta(1+t\sigma)$, so $\sigma\to u-1$.
For a corresponding eigenvector $x\to v$, put
$\widehat x=V^*x=y+z$, with $y=\mathcal A\widehat x$ and
$z=(I-\mathcal A)\widehat x$. Normalize $y=\one+h$ with
$\one^*h=0$. We start from $Px=\lambda x$. Substituting $x=V\widehat x$ and multiplying both sides by $\zeta^{-1}V^*$ gives
\begin{equation}
 [R+t\diag(\one+F)R(uK-I)+t\widehat G]\widehat x
       =(1+t\sigma)\widehat x.
 \label{eq:transformed-eigenvector}
\end{equation}
Every cycle of $R$ has $q$ states. Define
\[
 \mathcal A=\frac1q\sum_{k=0}^{q-1}R^k,\qquad
 (\mathcal Aw)_i=\frac1q\sum_{k=0}^{q-1}w_{R^k(i)}.
\]
Thus $\mathcal A$ replaces the coordinates on each cycle by their average. Its output is constant on each cycle, while the coordinates of $w-\mathcal Aw$ sum to zero on each cycle.

Put
\[
 W=\operatorname{Fix}(R)=\{w:Rw=w\}.
\]
These are precisely the vectors constant on each cycle. Vectors whose coordinates sum to zero on every cycle are orthogonal to $W$. Hence $\mathcal A$ is the orthogonal projection onto $W$, and
\[
 y=\mathcal A\widehat x\in W,\qquad
 z=(I-\mathcal A)\widehat x\in W^\perp.
\]

The operator $R-I$ maps $W^\perp$ into itself. Its kernel on this subspace is zero: a vector fixed by $R$ is constant on each cycle, and if its sum on each cycle is also zero, every coordinate vanishes. Hence $R-I$ is invertible on $W^\perp$.

Both $R$ and $K$ commute with $\mathcal A$. Consequently, $R(uK-I)$ maps $W$ into $W$ and $W^\perp$ into $W^\perp$. Applying $I-\mathcal A$ to \eqref{eq:transformed-eigenvector} gives
\[
 \begin{aligned}
 &[(R-I)+t(R(uK-I)-\sigma I)]z\\
 &\qquad=-t(I-\mathcal A)
 [\diag(F)R(uK-I)+\widehat G](y+z).
 \end{aligned}
\]
Since $\sigma$ remains bounded, the operator on the left remains invertible on $W^\perp$ for sufficiently small $t$, and its inverse has a fixed norm bound. The vector $y+z=\widehat x$ also remains bounded. Taking norms therefore gives
\[
 \|z\|\le Ct(\|F\|+\|G\|).
\]

Applying $\mathcal A$ to \eqref{eq:transformed-eigenvector} instead gives
\begin{equation}
 \begin{gathered}
 \sigma y=D(uK-I)y+\mathcal A\widehat Gy+\varepsilon,\\
 g=\mathcal AF,\qquad D=\diag(\one+g).
 \end{gathered}
 \label{eq:averaged-energy-equation}
\end{equation}
Here
\[
 \varepsilon=\mathcal A[\diag(F)R(uK-I)+\widehat G]z.
\]
The bound on $z$ yields
\[
 \|\varepsilon\|
 \le C(\|F\|+\|G\|)\|z\|
 \le Ct(\|F\|+\|G\|)^2.
\]

We next subtract the overall coordinate average from both sides of \eqref{eq:averaged-energy-equation}. Define
\[
 \Pi w=w-\frac{\one^*w}{n}\one,\qquad \Pi=I-J.
\]
The quantity $\one^*w/n$ is the average of the coordinates of $w$. Subtracting it from every coordinate gives a vector whose coordinate sum is zero.

We have $y=\one+h$, where $\one^*h=0$. Also, $g=\mathcal AF$ has coordinate sum zero because $F$ does and $\mathcal A$ preserves coordinate sums. Thus
\[
 \Pi y=h,\qquad \Pi h=h,\qquad \Pi g=g.
\]
Applying $\Pi$ to both sides of \eqref{eq:averaged-energy-equation}, expanding $D=\diag(\one+g)$, and using $K\one=\one$, we obtain
\[
 \begin{aligned}
 \sigma h={}&(u-1)g+(uK-I)h\\
 &+\Pi\!\left[
 \diag(g)(uK-I)h+\mathcal A\widehat Gy+\varepsilon
 \right].
 \end{aligned}
\]
Taking norms, using the boundedness of $\sigma$ and $y$, and inserting the bound on $\varepsilon$, gives
\[
 |u-1|\|g\|
 \le C\!\left(
 \|h\|+\|g\|\|h\|+\|G\|
 +t(\|F\|+\|G\|)^2
 \right).
\]

The cycle-averaging estimate \eqref{eq:cycle-averaging} gives $\|g\|\ge\sqrt\kappa\|F\|$. Since $\mathcal A$ is an orthogonal projection, we also have $\|g\|\le\|F\|$. Put
\[
 c_0=|u-1|\sqrt\kappa>0.
\]
The preceding inequality implies
\[
 c_0\|F\|
 \le C\!\left(
 \|h\|+\|G\|+\|F\|\|h\|
 +t(\|F\|+\|G\|)^2
 \right).
\]
Along the chosen sequence, $h,F,G,t\to0$. Hence, for sufficiently large indices,
\[
 C\|h\|+2Ct\|F\|\le\frac{c_0}{2},\qquad
 2t\|G\|\le1.
\]
Using $(\|F\|+\|G\|)^2\le2\|F\|^2+2\|G\|^2$, we obtain
\[
 c_0\|F\|
 \le C\|h\|+2C\|G\|+\frac{c_0}{2}\|F\|.
\]
Subtracting the last term from both sides gives
\[
 \|F\|\le C_1(\|h\|+\|G\|).
\]
Substituting this bound into the estimate for $\varepsilon$ yields
\begin{equation}
 \|\varepsilon\|
 \le Ct(\|F\|+\|G\|)^2
 \le C_2t(\|h\|^2+\|G\|^2).
 \label{eq:energy-error}
\end{equation}

Put $\mathfrak n(w)=\operatorname{Im}(w/(u-1))$.
On $\operatorname{Fix}(R)$, the eigenvalues of $K$ are the $d$th
roots of unity. Thus $u\rho$ ranges over $n$th roots including $u$
but excluding $1$. All the other $u\rho$ lie strictly inward of
the line through $1$ and $u$. Normality of $K$ and $\one^*h=0$
therefore give a fixed $c>0$ such that
\[
 \mathfrak n\bigl(y^*(uK-I)y\bigr)\le-c\|h\|^2.
\]
Multiply \eqref{eq:averaged-energy-equation} by $y^*D^{-1}$;
$D$ is positive for large $j$. Since $D^{-1}y\in\operatorname{Fix}(R)$,
the projection on the $G$ term drops out. Each outside entry
$G_{ik}$ contributes
\[
 G_{ik}\frac{\overline y_i}{1+g_i}
 \bigl(\zeta^{-1}\overline v_i v_k y_k-y_i\bigr).
\]
For each outside position $(i,k)$, the coefficient multiplying $G_{ik}$ has a strictly negative normal component at $y=\one,g=0$. Indeed, $\zeta^{-1}\overline v_i v_k$ is then an $n$th root other than $1,u$. Since there are finitely many such positions, continuity gives a constant $a>0$ such that every component remains at most $-a$ when $y$ and $g$ are sufficiently close to these limits. Since $G_{ik}\ge0$, the combined contribution is at most $-a\ell(G)$.
Using \eqref{eq:energy-error} and this bound, we obtain fixed $c,C>0$ with
\[
 \mathfrak n(\sigma)y^*D^{-1}y
 \le-c\|h\|^2-a\ell(G)
       +Ct\bigl(\|h\|^2+\ell(G)^2\bigr)\le0
\]
eventually. The factor $y^*D^{-1}y$ is positive, whereas the chosen
outside eigenvalues satisfy
$\operatorname{Im}((\lambda-\zeta)/(\zeta(u-1)))
=t\mathfrak n(\sigma)>0$. This contradiction proves
\eqref{eq:local-polygon}, since $H_n\subseteq\Om_{n,m}$.

It remains to show that each Karpelevi\v{c} arc meeting $\zeta$ leaves
the polygon immediately. For the other endpoint denominator $s$, expansion
of \eqref{eq:ito} gives
\[
 \begin{gathered}
 \lambda(\alpha)=\zeta(1+a\alpha+b\alpha^2+O(\alpha^3)),\\
 a=\frac{u-1}{q},\qquad
 \frac b{u-1}=\frac{u(n-s)}{dq^2}
             -\frac{(q-1)(u-1)}{2q^2}.
 \end{gathered}
\]
Consequently,
\[
 \operatorname{Im}\frac{\lambda(\alpha)-\zeta}{\zeta(u-1)}
 =\frac{\sin(2\pi/n)}{q^2}
  \left(\frac{q-1}{2}-\frac{n-s}{d}\right)\alpha^2
  +O(\alpha^3)>0
\]
for small positive $\alpha$. Indeed, $q+s>n$ gives $n-s\le q-1$,
so the bracket is at least $(q-1)(1/2-1/d)>0$ because $d\ge3$.
These arc points lie outside $H_n$ and hence, by the local equality,
outside $\Om_{n,m}$. Taking complex conjugates gives the arc on the
other side.
\end{proof}

\begin{proof}[Proof of the individual and arc assertions in Theorem~\ref{thm:main}]
\leavevmode\par
The upper bounds are given by \eqref{eq:upper}. For $q\ge2$, Proposition~
\ref{prop:nondivisor} gives the matching lower bound when $q$ does not
divide $n$, and Theorem~\ref{thm:divisor} gives it when $n=qd$ with
$d\ge3$. The Type~0 arcs already occur for $m=0$.

It remains to consider $n=2q$. The construction gives the whole arc
for $m=1$. Its other endpoint has denominator $q<s<n$: the Farey
relation excludes $s=2q$. Every proper divisor of $2q$ is at most $q$,
so $s$ does not divide $n$. This endpoint is therefore not an $n$th
root of unity and does not belong to $H_n=\Om_{n,0}$. Thus $m=0$
cannot contain the whole arc, and the smallest possible value is $m=1$.

The largest of these minimum values occurs at $q=2$ and equals
$\ceil{\floor{n/2}/2}=\ceil{(n-1)/4}$. Any region equal to $\Theta_n$
must contain every boundary arc. Conversely, the segment argument
following \eqref{eq:upper} fills the whole region from its boundary.
This proves \eqref{eq:individual-threshold}.

For $n=2$, every possible eigenvalue list is $\{1,t\}$ with
$-1\le t\le1$. It is the eigenvalue list of the stochastic circulant matrix
\[
 \frac12\begin{pmatrix}1+t&1-t\\1-t&1+t\end{pmatrix},
\]
so $m_2^*=0$. For $n=1$, the only stochastic matrix is $[1]$, and
$m_1^*=0$.
\end{proof}

\section{Complete spectra and the two-row problem}
\label{sec:joint}

To show that our family cannot reproduce every complete stochastic spectrum, we seek an eigenvalue list that occurs for a stochastic matrix of order $n$, but for no matrix in $\mathcal Q_{n,m}$. We record each characteristic polynomial by the ordered list of its coefficients, omitting the leading coefficient $1$:
\begin{equation}
\begin{gathered}
\det(zI-P)=z^n+c_{n-1}z^{n-1}+\cdots+c_0,\\
\operatorname{coeff}(P)=(c_{n-1},\ldots,c_0)\in\mathbb R^n.
\end{gathered}
\label{eq:coefficient-map}
\end{equation}
Each coefficient list determines the eigenvalues with multiplicity. The map $P\mapsto\operatorname{coeff}(P)$ is continuous, so the image of a compact matrix class is compact and hence closed. We first construct matrices from three-state cycles and at most two absorbing states. On each cycle, transitions from a state are possible only to itself or to the next state. The eigenvalues determine these matrices up to relabelling. For $n\ge4$, this lets us show that no matrix in our family with at least three interior rows has the same eigenvalues. Together with a separate argument at $n=3$, this settles the odd orders.

For orders $n=2q$ with odd $q\ge3$, we use the family shown in
Figure~\ref{fig:paired-family}. To show that another stochastic matrix
with the same eigenvalues must have the same transition pattern, we
compare traces of powers. These traces sum the probabilities of returning
to the starting state after a given number of steps.

We choose the transition probabilities so that no nonnegative matrix
with the same first and second traces can have a smaller fourth trace.
We prove that attaining this value requires each state to have exactly
one partner state, distinct from itself, with strictly positive
transition probabilities in both directions. The states therefore
divide into pairs. For several other powers, returns staying within
individual pairs already account for the entire trace, excluding
returns of those lengths that visit different pairs. Another trace
requires a return visiting different pairs to exist. Together, these
conditions force the cyclic arrangement of the pairs shown in
Figure~\ref{fig:paired-family}.

Within this arrangement, two consecutive Toeplitz rows impose
additional relations between the transition probabilities. We show
that matrices satisfying these relations produce a smaller-dimensional
family of characteristic polynomials. Consequently, some eigenvalue
lists from the Figure~\ref{fig:paired-family} family cannot occur in a
matrix with two consecutive Toeplitz rows.

For orders $n=8,12,16,\ldots$, write $n=2q+2$, where $q\ge3$ is odd.
We take a matrix $A$ of order $2q$ from the preceding construction and
form the direct sum
\[
 A\oplus C_2,\qquad C_2=\begin{pmatrix}0&1\\1&0\end{pmatrix}.
\]
This increases the order by two and adds the eigenvalues $1$ and $-1$.
We prove that every stochastic matrix with the resulting eigenvalues
must contain $C_2$ as a separate block. This allows us to extend the
exclusion argument to order $n$. Order four is treated separately by
an explicit matrix.

\subsection{Three homogeneous rows}

For $n\ge3$, write $n=3k+r$ with $r\in\{0,1,2\}$, and define
\begin{equation}
 b_i=\frac{i}{100n^3}\quad(1\le i\le k),\qquad
 J_n=I_r\oplus\bigoplus_{i=1}^k
 \bigl[b_iI_3+(1-b_i)C_3\bigr].
 \label{eq:rigid-witness}
\end{equation}
Here $\oplus$ denotes the direct sum of matrices, so there are no
transitions between different blocks.
In block $i$, the transition probabilities from each state are $b_i$
to itself and $1-b_i$ to the next state in the three-state cycle.
The block $I_r$ describes $r$ absorbing states.
When $r=0$, this block is omitted.

\begin{theorem}
\label{thm:rigidity}
Every stochastic matrix with the same eigenvalues as $J_n$, counted with
multiplicity, can be transformed into $J_n$ by applying the same permutation
to its rows and columns.
\end{theorem}
\begin{proof}
Let $P$ have the same eigenvalues as $J_n$, and put
$\delta=3\sum_i b_i<1$. Its only real eigenvalue is $1$.
Each irreducible diagonal block $B$ in Frobenius normal form therefore
has Perron eigenvalue $1$. It is substochastic; if $y>0$ is a left
Perron vector, then
\[
 y^{\mathsf T}(\one-B\one)=0.
\]
The vector in parentheses is nonnegative, so $B\one=\one$.
Every block is thus closed, and $P$ is block diagonal after relabelling.

Each block contributes one simple eigenvalue $1$, so there are $k+r$
blocks. A nonsingleton block has at least three states: a real
two-state stochastic block would contribute another real eigenvalue.
If $f$ blocks are singletons, counting states and taking the trace gives
\[
 3k+r\ge f+3(k+r-f),\qquad
 f\le\tr P=r+\delta<r+1.
\]
Hence $f=r$, and the other $k$ blocks all have order three.

Each three-state block $D$ has eigenvalue $1$ and one of the distinct
conjugate pairs
\[
 b_i+(1-b_i)\omega,\qquad b_i+(1-b_i)\overline\omega,
 \qquad \omega=e^{2\pi i/3}.
\]
Thus $\tr D=3b_i$ and $\tr D^2=3b_i^2$, giving
\[
 0=\sum_{j=1}^3(D_{jj}-b_i)^2
       +2\sum_{j<\ell}D_{j\ell}D_{\ell j}.
\]
All diagonal entries equal $b_i$, and opposite off-diagonal entries
cannot both be positive. Each row has at least one strictly positive
off-diagonal entry, since its off-diagonal entries sum to $1-b_i>0$.
There are three unordered pairs of states, and at most one transition
direction can have positive probability within each pair. Thus there
are exactly three strictly positive off-diagonal entries, one in each
row. They must form a directed three-cycle. Each transition to the
next state has probability $1-b_i$, while each state has probability
$b_i$ of remaining where it is. Hence
$D=b_iI_3+(1-b_i)C_3$ after relabelling, which proves the assertion.
\end{proof}

\begin{corollary}
\label{cor:three-rows}
If $n\ge4$ and $n-2m\ge3$, no matrix in $\Q_{n,m}$ has the same
eigenvalues as $J_n$, counted with multiplicity. Consequently,
\[
 j_n^*=\frac{n-1}{2}\quad\text{for odd }n\ge3,
 \qquad
 \frac n2-1\le j_n^*\le\frac n2\quad\text{for even }n\ge4.
\]
\end{corollary}
\begin{proof}
Suppose that $P\in\Q_{n,m}$ has the same eigenvalues as $J_n$, counted
with multiplicity. Its interior rows have equal diagonal entries because
they are cyclic shifts of rows of one circulant matrix. By
Theorem~\ref{thm:rigidity}, each value $b_i$ occurs on the diagonal
exactly three times, and $1$ occurs at most twice. There can therefore be
at most three interior rows.

If there are three, their states must form one of the closed three-state
cycles. The circulant matrix determined by these rows is then
\[
 C=b_\ell I+(1-b_\ell)C_n^s
\]
for some $\ell$ and $1\le s<n$. Thus, from each interior state $i$,
the chain either stays at $i$ or moves to $i+s$ modulo $n$. Since no
transition leaves the three states, adding $s$ modulo $n$ must preserve
\[
 I_{n,m}=\{m,m+1,\ldots,n-m-1\}.
\]
This is a nonempty proper interval in the cyclic order. Its first state,
$m$, is the only state in the interval whose predecessor lies outside it.
Any cyclic shift preserving the interval must therefore send $m$ to
itself. This requires $s=0$ modulo $n$, a contradiction.

For odd $n\ge5$, every $m<(n-1)/2$ leaves at least three interior
rows and is therefore excluded. At $m=(n-1)/2$, only one interior row
remains, so every stochastic matrix belongs to the family. For $n=3$,
one fixed state and a deterministic two-state cycle give the eigenvalues
$\{1,1,-1\}$. In a stochastic circulant matrix of order three, one
eigenvalue is $1$ and the other two are complex conjugates. They cannot
be the distinct real numbers $1$ and $-1$. Thus $j_3^*=1$ as well.

For even $n$, every $m<n/2-1$ leaves at least four interior rows and
is excluded, while $m=n/2$ leaves none and imposes no restriction.
This gives $n/2-1\le j_n^*\le n/2$.
\end{proof}

For even $n\ge4$, the matrices $J_n$ can nevertheless be relabelled into
$\Q_{n,n/2-1}$. Place three states of one block consecutively, with the
first two at the interior labels. Their rows agree with those of
$b_iI+(1-b_i)C_n$. Thus the three-row obstruction does not settle the
two-row problem.

\subsection{Two homogeneous rows in even order}

\begin{proposition}
\label{prop:four}
For $n=4$, we have $j_4^*=2$.
\end{proposition}
\begin{proof}
Choose $0<b<a<1$ with $3a-b>2$. The stochastic matrix
\[
 \begin{pmatrix}
  0&1&0&0\\
  1&0&0&0\\
  1-a&0&a&0\\
  1-b&0&0&b
 \end{pmatrix}
\]
has eigenvalues $\{1,-1,a,b\}$.

In any stochastic matrix with these eigenvalues, the eigenvalue $-1$
comes from a closed communicating class of states with even period.
A class of period four would also contribute
$i$ and $-i$. A class of period two splits into two groups, with
transitions only from one group to the other. On three states, such a
class would contribute the eigenvalue $0$; on four states, its other two
eigenvalues would have opposite signs. These possibilities conflict with
$a,b>0$.
The class must therefore consist of two states, with transition
probability $1$ from either state to the other. Since the eigenvalue $1$ has multiplicity one,
this is the only closed communicating class.

The submatrix on the other two states has the form
\[
 T=\begin{pmatrix}d_1&x\\y&d_2\end{pmatrix},\qquad
 d_1+d_2=a+b,\quad d_1d_2-xy=ab>0.
\]
Both $d_1$ and $d_2$ are positive, whereas the two cycle states have
diagonal entry $0$.

Suppose now that a matrix in $\Q_{4,1}$ has these eigenvalues. Its
interior rows, labelled $1$ and $2$, are cyclic shifts of one another
and have equal diagonal entries. Their states must therefore either
both belong to the two-state cycle or both be the other two states.
In the first case, row $1$ has its entry $1$ in column $2$, and row $2$
has its entry $1$ in column $1$. Shifting row $1$ one column forward
would instead put this entry in column $3$, a contradiction.

In the second case, $d_1=d_2=c=(a+b)/2$ and $xy=(a-b)^2/4$.
Row $1$ has entries $c$ and $x$ in columns $1$ and $2$. Its entry in
column $0$ equals the entry $y$ of row $2$ in column $1$, because the
rows are cyclic shifts. Since row $1$ sums to $1$,
\[
 1\ge c+x+y\ge\frac{a+b}{2}+2\sqrt{xy}
 =\frac{3a-b}{2}>1,
\]
again a contradiction. Thus $m=1$ cannot reproduce this eigenvalue
list, whereas $m=2$ imposes no restriction on the rows.
\end{proof}

For the remaining even orders, we seek complete eigenvalue lists of stochastic matrices
that no matrix in $\Q_{n,n/2-1}$ can reproduce.
We construct eigenvalue lists for which the traces of matrix powers
restrict the positions of positive entries in every nonnegative matrix
with those eigenvalues. Within this pattern, the lists compatible with
two rows that have equal diagonal entries and the same multiset of
entries form a set of smaller dimension than the full set of eigenvalue
lists. This leaves lists that our family cannot reproduce.

Two consecutive Toeplitz rows have equal diagonal entries and contain
the same entries, counted with repetitions. These conditions are
therefore necessary for a matrix to belong to our family.
To compare dimensions, we represent each eigenvalue list by the
coefficients of its monic characteristic polynomial.

The entry $(A^k)_{ii}$ is the sum of the products of matrix entries
along all walks of length $k$ that start and end at $i$.
Consequently, $\tr(A^k)$ sums these products over all starting states.
A directed cycle through $k$ distinct states, whose edge weights have
product $c$, contributes $kc$: one contribution for each possible
starting state on the cycle.

\begin{lemma}
\label{lem:paired-trace}
Let $n=2q$, $b>0$, and $0<\delta\le b/(100n^2)$. Put $w=b^2$.
Suppose a nonnegative matrix $A$ of order $n$ satisfies
\begin{equation}
 \tr A=q\delta,\qquad
 \tr A^2=q(\delta^2+2w),\qquad
 \tr A^4=q(\delta^4+4w\delta^2+2w^2).
 \label{eq:paired-traces}
\end{equation}
Then the $n$ row and column indices can be partitioned into $q$ pairs.
For each pair $\{i,j\}$, put $d=A_{ii}$ and $e=A_{jj}$. We have
\begin{equation}
 d+e=\delta,\qquad A_{ij}A_{ji}=w+de.
 \label{eq:matched-pair}
\end{equation}
If $i$ and $j$ belong to different pairs, then $A_{ij}A_{ji}=0$.
Draw an arrow $i\to j$ whenever $A_{ij}>0$. There is no directed
cycle through four distinct indices. A directed cycle through three
distinct indices can occur only when all three corresponding diagonal
entries are $0$.
\end{lemma}
\begin{proof}
Choose $u,v>0$ such that $u-v=\delta$ and $uv=b^2$. First we show
that real numbers $\mu_1,\ldots,\mu_n$ satisfying
\[
 \sum_i\mu_i=q\delta,\qquad
 \sum_i\mu_i^2=q(\delta^2+2b^2)
\]
obey
\begin{equation}
 \sum_i\mu_i^4\ge q(\delta^4+4b^2\delta^2+2b^4),
 \label{eq:scalar-fourth}
\end{equation}
with equality precisely when the list consists of $q$ copies of $u$
and $q$ copies of $-v$.

Minimise $\sum_i\mu_i^4$ subject to the two constraints. The feasible
set is compact and contains the proposed list. Comparing a minimiser
with that list gives
\[
 \sum_i\left(\mu_i^2-b^2-\frac{\delta^2}{2}\right)^2
 \le n\delta^2\left(b^2+\frac{\delta^2}{4}\right).
\]
In particular, each $\mu_i$ satisfies
$\bigl||\mu_i|-b\bigr|\le2\sqrt n\,\delta<b/50$.
Let $n_+$ and $n_-$ count the positive and negative entries. Then
\[
 |(n_+-n_-)b|\le n\delta/2+2n\sqrt n\,\delta<b,
\]
so $n_+=n_-=q$. The entries are not all equal, so the gradients of
the two constraints are independent. Lagrange multipliers give
\[
 4\mu_i^3-2\eta\mu_i-\xi=0
\]
for every $i$. If three distinct values occurred, they would be the
three roots of this cubic and would sum to $0$. But three numbers each
within $b/50$ of either $b$ or $-b$ cannot sum to $0$. Thus there are
exactly two values, one positive and one negative, each occurring $q$
times. From the constraints, it follows that they are $u$ and $-v$,
proving \eqref{eq:scalar-fourth}.

Now define the symmetric nonnegative matrix
\[
 B_{ij}=\sqrt{A_{ij}A_{ji}}.
\]
We have $\tr B=\tr A$ and $\tr B^2=\tr A^2$, so the real
eigenvalues of $B$ satisfy the two constraints above. For a closed walk
$i\to j\to k\to\ell\to i$, write
\[
 p=A_{ij}A_{jk}A_{k\ell}A_{\ell i},\qquad
 p_{\mathrm{rev}}=A_{i\ell}A_{\ell k}A_{kj}A_{ji}
\]
for the products along the walk and its reverse. The corresponding
product for $B$ satisfies the AM--GM inequality
\[
 B_{ij}B_{jk}B_{k\ell}B_{\ell i}
 =\sqrt{p\,p_{\mathrm{rev}}}
 \le\frac{p+p_{\mathrm{rev}}}{2}.
\]
Summing over all ordered quadruples $(i,j,k,\ell)$ gives
$\tr B^4\le\tr A^4$, since reversing the walks leaves the sum
of the products for $A$ unchanged. By \eqref{eq:paired-traces},
the right-hand side equals the lower bound in
\eqref{eq:scalar-fourth}. We therefore have equality, and the eigenvalues
of $B$ are $u$ and $-v$, each with multiplicity $q$. Every connected component of $B$ has simple Perron eigenvalue $u$;
all its remaining eigenvalues are $-v$. If its order is $r$, its
nonnegative trace gives
\[
 0\le u-(r-1)v,\qquad r\le1+u/v<3.
\]
Thus every component has one or two indices. Each two-index component
contributes one eigenvalue $-v$; the $q$ negative eigenvalues and $2q$
indices force exactly $q$ pairs.
Since $B_{ii}=A_{ii}$ and
$B_{ij}^2=A_{ij}A_{ji}$, the trace and determinant on each pair give
\eqref{eq:matched-pair}. Between different pairs, $B_{ij}=0$ and hence
$A_{ij}A_{ji}=0$.

Each closed walk contributing to $\tr A^4$ contributes the product
of its four matrix entries. Walks staying within a single pair already
contribute $q(u^4+v^4)$, the whole prescribed trace. Note that a directed cycle
through four distinct indices would add a positive contribution, as
would a directed cycle through three distinct indices followed by a
loop with positive diagonal entry. Therefore both are excluded.
\end{proof}

\begin{lemma}
\label{lem:paired-ladder}
Let $q=3$ or $q\ge5$, put $n=2q$, and use the parameters of
Lemma~\ref{lem:paired-trace}. Define $u,v>0$ by
$u-v=\delta$ and $uv=w$, and put
\[
 a_k=q(u^k+(-v)^k),\qquad R=1-\delta.
\]
Assume $w<R$. For a stochastic matrix $P$ of order $n$, write
$s_k=\tr(P^k)$. Suppose that
\begin{equation}
 s_k=a_k\quad(1\le k<q),\qquad s_q>a_q,
 \qquad s_{q+1}=a_{q+1}.
 \label{eq:ladder-traces}
\end{equation}
After applying the same permutation to its rows and columns, we can
group the states into $q$ pairs. Label the two states in pair $i$ by
$x_i$ and $y_i$, $0\le i<q$. Then
\begin{equation}
 \begin{array}{lll}
 P_{x_i x_i}=\delta,&P_{x_i y_i}=r_i,&P_{x_i x_{i+1}}=R-r_i,\\
 P_{y_i x_i}=w/r_i,&P_{y_i y_{i+1}}=1-w/r_i.&
 \end{array}
 \label{eq:paired-ladder}
\end{equation}
Here $i+1$ is read modulo $q$, all other entries are $0$, and
\begin{equation}
 w<r_i\le R,\qquad\prod_i(R-r_i)=0.
 \label{eq:ladder-domain}
\end{equation}
Thus at least one entry $P_{x_i x_{i+1}}$ is $0$.
Conversely, every matrix of this form has the specified traces.

Every nonnegative matrix of order $n$ satisfying
\eqref{eq:ladder-traces} is irreducible: its directed graph has a path
from every index to every other index.
\end{lemma}
\begin{proof}
Let $P$ be a nonnegative matrix of order $n$ whose traces
$s_k=\tr(P^k)$ satisfy \eqref{eq:ladder-traces}. These conditions contain all three equalities
required by Lemma~\ref{lem:paired-trace}. The one for $\tr(P^4)$ comes
from $s_{q+1}=a_{q+1}$ when $q=3$, and from $s_4=a_4$ when $q\ge5$.
The lemma therefore partitions the indices into $q$ pairs. Each
corresponding two-by-two block has characteristic polynomial
$L(z)=z^2-\delta z-w$, so closed walks staying within individual pairs
contribute $a_k$ to $\tr(P^k)$.

Consequently, no directed cycle visiting states from more than one pair has length
less than $q$, and at least one such cycle has length $q$. The last
equality in \eqref{eq:ladder-traces} excludes every cycle of length
$q+1$. It also excludes a cycle of length $q$ through an index with
positive diagonal entry: inserting a loop at that index would give an
additional closed walk of length $q+1$.

Each pair has diagonal entries summing to $\delta>0$, so at most one
of them is $0$. A cycle of length $q$ must therefore visit one index
with diagonal entry $0$ from each pair. Label these indices
$y_0,\ldots,y_{q-1}$ in the order of the cycle, and call their partners
$x_i$. Then $P_{y_i y_i}=0$ and $P_{x_i x_i}=\delta$.

Consider an arrow between different pairs, and write its destination
pair as $i+d$, where $1\le d<q$ and pair indices are read modulo $q$.
Complete this arrow to a cycle by following the cycle of $y$ indices
and using the arrows $x_j\leftrightarrow y_j$ when needed. The resulting
lengths are
\[
 \begin{array}{c|c}
 \text{arrow}&\text{length of the resulting cycle}\\ \hline
 y_i\to y_{i+d}&q-d+1\\
 x_i\to y_{i+d}\text{ or }y_i\to x_{i+d}&q-d+2\\
 x_i\to x_{i+d}&q-d+3.
 \end{array}
\]
Every arrow in the table except $y_i\to y_{i+1}$ and
$x_i\to x_{i+1}$ would produce a cycle of length less than $q$, a
cycle of length $q+1$, or a cycle of length $q$ through a positive
diagonal entry. A cycle of length less than $q$ would contradict
$s_k=a_k$ for the corresponding length $k$. A cycle of length $q+1$
would contradict $s_{q+1}=a_{q+1}$. Finally, a cycle of length $q$
through a state with positive diagonal entry would also contradict
this last equality: inserting a transition from that state to itself
produces an additional positive contribution to $s_{q+1}$. A cycle
through all the $x_i$ would also have length $q$ and meet positive
diagonal entries, so at least one arrow $x_i\to x_{i+1}$ is absent.

The arrows $x_i\leftrightarrow y_i$ and the directed cycle through the
$y_i$ give a path between every two indices. Hence $P$ is irreducible.
If $P$ is stochastic, put
$r_i=P_{x_i y_i}$. The row sums give \eqref{eq:paired-ladder}.
The positive entries on the $y$ cycle and the missing arrow on the $x$
cycle give \eqref{eq:ladder-domain}.

Conversely, in a matrix of this form, any directed cycle that visits
states from more than one pair must visit all $q$ pairs in cyclic order. The cycle through the $y_i$ has
length $q$ and all its diagonal entries are $0$. Any other such cycle
must use an arrow from a $y$ index to an $x$ index and another arrow
back to a $y$ index, so its length is at least $q+2$. Thus closed walks
that stay within one pair account for all traces below $q$ and for the
trace at $q+1$, while
the positive cycle through the $y_i$ adds a strictly positive
contribution to the trace at $q$.
\end{proof}

\begin{figure}[htbp]
\centering
\includegraphics{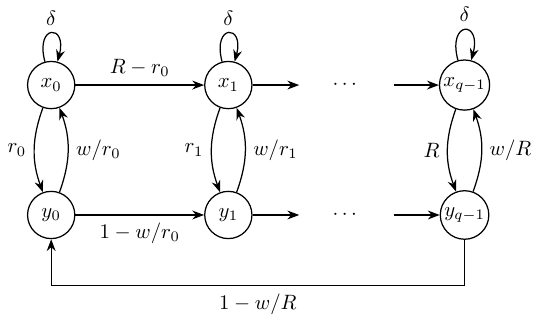}
\caption{We choose $r_{q-1}=R=1-\delta$, so there is no transition
from $x_{q-1}$ to $x_0$. Arrows represent transitions, and their labels
give the corresponding probabilities. The horizontal probabilities are
$R-r_i$ for $x_i\to x_{i+1}$ and $1-w/r_i$ for $y_i\to y_{i+1}$,
with indices read modulo $q$. The states $y_i$ form a directed cycle.
Within each pair $\{x_i,y_i\}$, the probabilities in the two directions
have product $w$. The trace conditions in \eqref{eq:ladder-traces}
force every stochastic matrix satisfying them to have this pattern
after relabelling, with at least one transition $x_i\to x_{i+1}$ absent.}
\label{fig:paired-family}
\end{figure}

We now fix $r_{q-1}=R$ in \eqref{eq:paired-ladder} and let the other
$q-1$ parameters vary with $w<r_i<R$. The only absent transition
between successive $x$ states is then $x_{q-1}\to x_0$, as in
Figure~\ref{fig:paired-family}. To prove that the corresponding characteristic
polynomials form a family of dimension $q-1$, we show that varying the free
parameters produces $q-1$ independent changes in the polynomial coefficients
at a suitable parameter choice.

\begin{samepage}
\begin{lemma}
\label{lem:ladder-dimension}
Consider the matrices $P$ in \eqref{eq:paired-ladder} with
\[
 r_{q-1}=R,\qquad w<r_i<R\quad(0\le i<q-1).
\]
At a suitable choice of $r_0,\ldots,r_{q-2}$, the $q-1$ vectors
\[
 \frac{\partial\,\operatorname{coeff}(P)}{\partial r_0},
 \ldots,
 \frac{\partial\,\operatorname{coeff}(P)}{\partial r_{q-2}}
\]
are linearly independent. Thus the characteristic coefficient vectors
contain a locally smooth family with $q-1$ independent real coordinates.
\end{lemma}
\end{samepage}
\begin{proof}
Replace $r_0,\ldots,r_{q-2}$ by the positive parameters
\[
 \alpha_i=\frac{r_{i+1}(R-r_i)}{r_i-w},\qquad
 r_i=\frac{Rr_{i+1}+w\alpha_i}{r_{i+1}+\alpha_i}
 \qquad(0\le i<q-1).
\]
Starting from $r_{q-1}=R$, use the second formula successively for
$i=q-2,q-3,\ldots,0$ to obtain $r_{q-2},r_{q-3},\ldots,r_0$.
These values satisfy $w<r_i<R$, so this is an invertible change of
parameters. Put $\alpha_{q-1}=0$, $\beta_i=1-w/r_i$ and
$E=\prod_i\beta_i$. The eigenvector equations are
\[
 (z-\delta)X_i=r_iY_i+(R-r_i)X_{i+1},\qquad
 zY_i=(w/r_i)X_i+\beta_iY_{i+1}.
\]
We first eliminate the variables $X_i$. The second eigenvector
equation gives
\[
 X_i=\frac{r_i}{w}\bigl(zY_i-\beta_iY_{i+1}\bigr).
\]
Substitute these expressions into the first eigenvector equations and
multiply equation $i$ by $w/r_i$. With $L(z)=z^2-\delta z-w$,
this gives
\[
 L(z)Y_i-\beta_i\bigl((1+\alpha_i)z-\delta\bigr)Y_{i+1}
       +\alpha_i\beta_i\beta_{i+1}Y_{i+2}=0.
\]
We have therefore replaced the original equations in $2q$ variables
by $q$ equations in $Y_0,\ldots,Y_{q-1}$. Setting
\[
 t_i=\prod_{j<i}\beta_j,\qquad U_i=t_iY_i
\]
and multiplying equation $i$ by $t_i$ gives
\[
 L(z)U_i-\bigl((1+\alpha_i)z-\delta\bigr)U_{i+1}
       +\alpha_iU_{i+2}=0,\qquad U_{i+q}=EU_i.
\]
Define the backward recurrence
\begin{equation}
 \begin{gathered}
 K_q=1,\qquad K_{q-1}=z-\delta,\\
 K_i=((1+\alpha_i)z-\delta)K_{i+1}-\alpha_iLK_{i+2}
 \qquad(i=q-2,\ldots,0).
 \end{gathered}
 \label{eq:ladder-recurrence}
\end{equation}
Then
\begin{equation}
 f(z)=\det(zI-P)=L(z)^q-EK_0(z).
 \label{eq:ladder-polynomial}
\end{equation}
Let $M(z)$ be the $q\times q$ coefficient matrix of the equations in
$U_0,\ldots,U_{q-1}$. Its diagonal entries are $L(z)$, its first
superdiagonal entries are $-((1+\alpha_i)z-\delta)$, and its second
superdiagonal entries are $\alpha_i$. The relation $U_q=EU_0$ gives
two additional entries in the first column:
\[
 M_{q-2,0}=E\alpha_{q-2},\qquad
 M_{q-1,0}=-E(z-\delta).
\]
To track the determinant, let $N(z)$ be the coefficient matrix of the
equations in $Y_i$ immediately after substitution, before multiplying
equation $i$ by $w/r_i$. Eliminating the $X_i$ from $zI-P$, using
their nonzero coefficients $-w/r_i$, gives
\[
 \det(zI-P)=\left(\prod_i\frac{w}{r_i}\right)\det N(z).
\]
The signs from these negative coefficients cancel the sign from
exchanging the two groups of rows. Multiplying equation $i$ by
$w/r_i$ incorporates the displayed product into the determinant of
the remaining matrix. Finally, substituting $Y_i=U_i/t_i$ divides
column $i$ by $t_i$, while multiplying equation $i$ by $t_i$ cancels
that factor. Hence
\[
 \det(zI-P)=\det M(z).
\]
Setting $M_{q-2,0}$ and $M_{q-1,0}$ to $0$ gives an upper triangular matrix.
Its determinant is $L^q$; backward substitution and the rank-one determinant identity
give \eqref{eq:ladder-polynomial} whenever $L(z)\ne0$. Both sides
are polynomials in $z$, so the formula holds for every $z$, including
the roots of $L$.

The recurrence makes $K_0$ a monic polynomial of degree $q$, with next coefficient
$-q\delta$. Its coefficient map has rank $q-1$ at positive points
arbitrarily close to $0$. For the induction, the case $q=2$ is
$K_0=(z-\delta)^2+w\alpha_0$.
At $\alpha_0=0$, we have $K_0=(z-\delta)K_1$. The derivatives in
$\alpha_1,\ldots,\alpha_{q-2}$ are independent by induction and vanish
at $z=\delta$, whereas
\[
 \left.\frac{\partial K_0}{\partial\alpha_0}\right|_{z=\delta}
 =\delta K_1(\delta)+wK_2(\delta)>0.
\]
Indeed, $K_i(\delta)=\alpha_i[\delta K_{i+1}(\delta)+wK_{i+2}(\delta)]$
is positive for $i\le q-2$ when the parameters are positive.
The new derivative is independent, and continuity preserves this
independence for small positive $\alpha_0$.

Finally, stochasticity gives $E=L(1)^q/K_0(1)$. Conversely,
$f$ determines $EK_0=L^q-f$; dividing by its positive leading
coefficient recovers the monic polynomial $K_0$. Passing to
characteristic coefficients therefore preserves the rank $q-1$,
which is also the number of free parameters.
\end{proof}

\begin{theorem}
\label{thm:even-joint}
For every even $n\ge4$, there is a stochastic matrix of order $n$
for which no matrix in $\Q_{n,n/2-1}$ has the same eigenvalues,
counted with multiplicity. Hence $j_n^*=n/2$.
\end{theorem}
\begin{proof}
Proposition~\ref{prop:four} handles $n=4$. Suppose first that $n=2q$
with $q\ge3$ odd. Fix $b=1/2$ and choose
$0<\delta\le1/(200n^2)$. By Lemma~\ref{lem:paired-ladder}, every
stochastic matrix with the same eigenvalues as a matrix in
\eqref{eq:paired-ladder} again has that form after relabelling.

Rows labelled $x_i$ have diagonal entry $\delta$, while those labelled
$y_i$ have diagonal entry $0$. Two rows with equal diagonal entries
and the same multiset of entries must therefore both be $x$ rows or
both be $y$ rows. For distinct $i,j$,
their other entries give the following necessary and sufficient
conditions:
\begin{equation}
 \begin{array}{ll}
 x\text{ rows}:& r_i=r_j\quad\text{or}\quad r_i+r_j=R,\\
 y\text{ rows}:& r_i=r_j\quad\text{or}\quad
 r_ir_j=w(r_i+r_j).
 \end{array}
 \label{eq:equal-ladder-profiles}
\end{equation}
A subset of $\mathbb R^k$ is called \emph{semialgebraic} if it is
a finite union of sets defined by finitely many polynomial equations
and inequalities. A map is semialgebraic if its graph is such a set;
see \cite{BasuPollackRoy2006}.

Rotate the pair labels so that an absent $x$-transition is last.
Then $r_{q-1}=R$, and the other $q-1$ parameters satisfy
$w<r_i\le R$. Each condition in \eqref{eq:equal-ladder-profiles}
is a proper polynomial equation on this domain. Their finite union
has dimension at most $q-2$, as does its coefficient image under
the rational, hence semialgebraic, coefficient map.
Lemma~\ref{lem:ladder-dimension} gives a family of dimension $q-1$,
so some of its eigenvalue lists cannot occur in any stochastic
matrix with two rows having equal diagonal entries and the same
multiset of entries. They therefore cannot occur
in $\Q_{n,n/2-1}$.

For an order divisible by four, write $n=2q+2$ with $q\ge3$ odd.
At $\delta=0$ and $\alpha=0$, the order-$2q$ family has polynomial
\[
 f_0(z)=(z^2-w)^q-(1-w)^qz^q.
\]
The polynomials $f_0(z)$ and $f_0(-z)$ have no common root: since
$q$ is odd, their difference is $-2(1-w)^qz^q$, whereas
$f_0(0)=(-w)^q\ne0$.
Thus this polynomial has no zero root and no pair of roots $z,-z$.
These properties persist for small positive $\delta$ and $\alpha$.
For each fixed positive $\delta$, the points where the coefficient
map has full rank in Lemma~\ref{lem:ladder-dimension} occur arbitrarily
close to $0$. We can therefore choose a matrix $A$ in this neighbourhood
so that no stochastic matrix with its eigenvalues has two rows with
equal diagonal entries and the same multiset of entries.

Add the eigenvalues $1,-1$ by placing $A$ and the deterministic
two-state cycle $C_2$ along the diagonal of a larger matrix. Denote
the resulting eigenvalue list by $\Lambda$. Let $Q$ be any stochastic
matrix with eigenvalue list $\Lambda$. The eigenvalue $-1$ belongs to
a closed irreducible class of even period $h$. Such a class has every
$h$th root of unity as an eigenvalue. On the other hand, $A$ is
irreducible and has a strictly positive diagonal entry, so its only
eigenvalue on the unit circle is $1$. Thus the only eigenvalues of $Q$
on the unit circle are $1$ and $-1$. It follows that $h=2$.
The eigenvalues of this class
occur in pairs $z,-z$. Since $\Lambda$ has no zero eigenvalue and no
other such pair, the class has two states and its matrix is $C_2$.

Deleting those two rows and columns leaves a matrix $M$ of order $2q$
with row sums at most $1$ and the same eigenvalues as $A$.
Lemma~\ref{lem:paired-ladder} makes $M$ irreducible. Since its spectral
radius is $1$, a positive left Perron vector annihilates the nonnegative
vector $\one-M\one$, where $\one$ is the all-ones vector.
Hence this vector is $0$, and $M$ is stochastic. No transition joins
the remaining states to the two-state cycle, so the whole matrix has
blocks $M$ and $C_2$ along its diagonal after relabelling.

By the choice of $A$, no two rows of $M$ have equal diagonal entries
and the same multiset of entries.
Each row of $M$ has at least two positive entries, whereas each row of
$C_2$ has only one. The two interior rows could therefore only come
from $C_2$. At consecutive labels, these rows have displacements
$1$ and $-1$ modulo $n$, which cannot agree for $n>2$.
Thus no matrix in $\Q_{n,n/2-1}$ has eigenvalue list $\Lambda$.
Finally, $m=n/2$ leaves no interior rows and allows every stochastic
matrix in each even order.
\end{proof}

\begin{corollary}
\label{cor:positive-joint-witnesses}
For every even $n\ge4$, there is a stochastic matrix of order $n$
with strictly positive rational entries and $n$ distinct eigenvalues
such that no matrix in $\Q_{n,n/2-1}$ has the same eigenvalues.
\end{corollary}
\begin{proof}
By Theorem~\ref{thm:even-joint}, choose a stochastic matrix whose
coefficient vector lies outside the image of $\Q_{n,n/2-1}$.
That image is compact, so all sufficiently small stochastic
perturbations remain excluded. Mixing with $J$ and then approximating
the rows by rational probability vectors gives a strictly positive
rational excluded matrix $Q$.

Consider $Q_\varepsilon=(1-\varepsilon)Q+\varepsilon C_n$.
Its characteristic-polynomial discriminant is a polynomial in
$\varepsilon$, nonzero at $\varepsilon=1$ because $C_n$ has distinct
eigenvalues. It has only finitely many zeros. Arbitrarily small
positive rational $\varepsilon$ therefore give strictly positive
rational matrices with distinct eigenvalues that remain excluded.
\end{proof}

\section{Boundary eigenvalues and arc exclusion}
\label{sec:minimal-arcs}

Proposition~\ref{prop:nondivisor} and Theorem~\ref{thm:local-polygon}
excluded points sufficiently close to an endpoint of a Karpelevi\v{c}
arc. We now consider an arc whose points between the endpoints first
occur as eigenvalues of stochastic matrices at order $n\ge4$.
For arcs whose smaller endpoint denominator is at least $2$, we give
bounds on $m$ that exclude every point between the endpoints from
$\Om_{n,m}$. We first prove that any one of these eigenvalues
determines the characteristic polynomial of every stochastic matrix
of order $n$ that has it.

The Type~III and Type~II classifications in
\cite{VerbekenGinisTypeIII2026,VerbekenGinisTypeII2026} describe stochastic
matrices with a prescribed characteristic polynomial. Here the polynomial
is first forced by the boundary eigenvalue. We then use those
classifications to count the rows with two positive entries.

\begin{theorem}[A boundary eigenvalue determines the characteristic polynomial]
\label{thm:boundary-polynomial}
\label{lem:boundary-polynomial}
Let $\lambda$ be a nonreal point between the endpoints of a
Karpelevi\v{c} boundary arc. Let $n\ge4$ be the smallest order of a
stochastic matrix with eigenvalue $\lambda$, so
$\lambda\in\Theta_n\setminus\Theta_{n-1}$. Write the endpoint
denominators as $q<s$ and put $d=\lfloor n/q\rfloor$.
Let $0<\alpha<1$ be the arc parameter: for $q\ge2$, $\lambda$ is
a selected root of \eqref{eq:ito}; for $q=1$,
$\lambda=(1-\alpha)+\alpha e^{\pm2\pi i/n}$. Put $\beta=1-\alpha$.
Every stochastic matrix $P$ of order $n$ with eigenvalue $\lambda$
has characteristic polynomial
\begin{equation}
 z^{n-qd}(z^q-\beta)^d-\alpha^d z^{n-s}.
 \label{eq:rigid-boundary-polynomial}
\end{equation}
In Type III, where $s=n>qd$, replace $\lambda$ by its complex
conjugate if necessary so that the endpoint fractions are ordered
$p/q<r/n$. Label the states $0,\ldots,n-1$ in the counterclockwise
order of the coordinates of an eigenvector for $\lambda$, and write
\[
 R(i)=i+r\pmod n,\qquad S(i)=i+r-1\pmod n.
\]
In this labelling, $P_{i,R(i)}>0$ for every $i$, and every other entry
of row $i$ is $0$ except possibly $P_{i,S(i)}$.
\end{theorem}

We use Dmitriev and Dynkin's theorem on the polygon formed by an
eigenvector's coordinates. Besides identifying the two possible nonzero
entries in each row, we need its conclusion that the coefficient $a_i$
in the next lemma is positive for every $i$.

\begin{lemma}[Dmitriev--Dynkin contact theorem]
\label{lem:positive-contact}
Let $\lambda\in\Theta_n\setminus\Theta_{n-1}$ be nonreal, and suppose
$t\lambda\notin\Theta_n$ for every $t>1$. This means that $\lambda$ is
the farthest point of $\Theta_n$ on the ray from the origin through
$\lambda$. Let $K$ be a convex polygon
with $n$ vertices satisfying $\lambda K\subseteq K$.
After replacing $\lambda$ and every vertex of $K$ by their complex
conjugates if necessary, label the vertices $v_0,\ldots,v_{n-1}$
counterclockwise. There is an integer $1\le\kappa<n$ such that, for
every $i$,
\begin{equation}
 \lambda v_{i-\kappa}
 =\gamma_i v_{i-1}+a_i v_i\in(v_{i-1},v_i],
 \qquad 0\le\gamma_i<1,\quad a_i=1-\gamma_i>0.
 \label{eq:positive-contact}
\end{equation}
All indices are read modulo $n$. Thus $\lambda v_{i-\kappa}$ lies
on the segment from $v_{i-1}$ to $v_i$, including $v_i$ and excluding
$v_{i-1}$. Each of these $n$ segments contains exactly one of the
$n$ points $\lambda v_j$.

In particular, let $v$ be an eigenvector for $\lambda$ of a stochastic
matrix $A$ of order $n$, and take $K$ to be the convex hull of its
coordinates. Relabel the states in the vertex order above, applying
the same permutation to the rows and columns of $A$. Then
\[
 A=C^\kappa B(\gamma),\qquad
 (Cx)_i=x_{i+1},\qquad
 (B(\gamma)x)_i=\gamma_i x_{i-1}+a_i x_i.
\]
\end{lemma}

The conclusion of the Dmitriev--Dynkin contact theorem that
$\lambda v_{i-\kappa}$ lies on the edge from $v_{i-1}$ to $v_i$
is part 5.1 of the Basic Theorem in
\cite[\S2]{DmitrievDynkin1946}. Its footnote gives the strict positivity
used here, and its proof in \S13 shows that each segment, with one
endpoint excluded, contains exactly one image vertex. An English
translation appears in \cite[Appendix A, pp.~A-5 and A-16]{Swift1972}.
The footnote rules out $\gamma_i=0$ and $a_j=0$ occurring for any
indices $i,j$. Thus either all $\gamma_i$ are positive or all $a_i$
are positive. Choosing between the polygon and its complex conjugate
gives $a_i>0$ for every $i$. The theorem applies to every polygon $K$
satisfying the lemma's hypotheses; we do not need to choose one of
minimum area.

Now let $v$ be an eigenvector for $\lambda$ of a stochastic matrix $A$
of order $n$, and let $K$ be the convex hull of its coordinates.
The stochastic row equations
\[
 \lambda v_i=\sum_j A_{ij}v_j,\qquad
 A_{ij}\ge0,\qquad \sum_j A_{ij}=1
\]
give $\lambda v_i\in K$ for every $i$, hence $\lambda K\subseteq K$
by convexity. The hull $K$ must have $n$ vertices.
Otherwise, expressing $\lambda$ times each vertex as a convex combination
of the vertices would give a stochastic matrix of smaller order with
eigenvalue $\lambda$, contrary to $\lambda\notin\Theta_{n-1}$.
Thus all $n$ coordinates are distinct vertices of this hull.
Every convex representation of $\lambda v_{i-\kappa}$ by these vertices
uses only $v_{i-1}$ and $v_i$, with uniquely determined coefficients.
Because $Av=\lambda v$, the rows of $A$ must therefore have exactly
the coefficients displayed in the lemma.

The next lemma computes a characteristic polynomial from the directed
cycles in the graph of a matrix.

\begin{lemma}
\label{lem:return-determinant}
Let $q,s,d$ be positive integers with $q<s\le n$ and $qd\le n$.
Consider a directed graph on $n$ vertices with nonnegative edge weights,
and let $A$ be its adjacency matrix. The weight of a cycle is the product
of its edge weights. A cycle is simple if it does not repeat a vertex,
except when returning to its starting vertex.

Suppose every simple directed cycle with positive weight is one of
$C_1,\ldots,C_d,D$, where the $C_j$ have length $q$ and no two share
a vertex, and $D$ has length $s$ and shares at least one vertex with
each $C_j$. Write $b_j$ for the weight of $C_j$, and suppose the
weight of $D$ is $\prod_{j=1}^d(1-b_j)$. Some $b_j$ may be zero.
Then $A$ has characteristic polynomial
\begin{equation}
 z^{n-qd}\prod_{j=1}^d(z^q-b_j)
 -z^{n-s}\prod_{j=1}^d(1-b_j).
 \label{eq:return-determinant}
\end{equation}
\end{lemma}
\begin{proof}
We use the permutation expansion of $\det(zI-A)$, which selects
one entry from each row and each column. When a diagonal entry
$z-A_{ii}$ is selected, expand it as the two choices $z$ and $-A_{ii}$.
In each resulting nonzero term, the selected entries from $-A$ form
directed cycles with no shared vertex: each vertex involved has exactly
one selected outgoing edge and one selected incoming edge. Vertices
outside these cycles contribute factors $z$.

A selected cycle of length $\ell$ has permutation sign
$(-1)^{\ell-1}$, while its $\ell$ entries from $-A$ contribute a sign
$(-1)^\ell$. Together these give one minus sign for the cycle,
regardless of its length. Its remaining factor is the product of its
edge weights.

First select any subset $J\subseteq\{1,\ldots,d\}$ of the cycles $C_j$.
They occupy $q|J|$ vertices and have total contribution
\[
 (-1)^{|J|}z^{n-q|J|}\prod_{j\in J}b_j.
\]
Summing over all choices gives
\[
 \sum_{J\subseteq\{1,\ldots,d\}}
 (-1)^{|J|}z^{n-q|J|}\prod_{j\in J}b_j
 =z^{n-qd}\prod_{j=1}^d(z^q-b_j).
\]
This is the first term of \eqref{eq:return-determinant}: the $j$th
factor chooses $z^q$ when $C_j$ is omitted and $-b_j$ when it is included.

If $D$ is selected, none of the $C_j$ can also be selected, because
$D$ shares a vertex with each of them. The cycle $D$ occupies $s$
vertices, so its contribution is
\[
 -z^{n-s}\prod_{j=1}^d(1-b_j).
\]
By hypothesis there are no other positive-weight simple cycles.
These choices therefore account for every nonzero term in the
determinant expansion and give the full characteristic polynomial,
including algebraic multiplicities.
\end{proof}

\begin{lemma}
\label{lem:isospectral-return-form}
Under the hypotheses of Theorem~\ref{thm:boundary-polynomial}, let $P$ be
any stochastic matrix of order $n$ with eigenvalue $\lambda$.
There is a stochastic matrix $A$ with the same characteristic polynomial
as $P$ whose positive-weight simple cycles are among
$C_1,\ldots,C_d,D$. The cycles $C_j$ have length $q$ and pairwise
disjoint vertex sets, with weights $0\le b_j<1$; zero weights are allowed.
The cycle $D$ has length $s$, meets every $C_j$, and has weight
$\prod_j(1-b_j)$.
After complex conjugation if necessary, write the endpoint fractions as
$p/q<r/s$, with $q<s$, and let $\lambda=\rho e^{i\theta}$, where
$2\pi p/q<\theta<2\pi r/s$. Put
\[
 \phi_j=\operatorname{Arg}(\lambda^q-b_j)\in(0,\pi).
\]
The return paths give the identity in
\cite[Theorem~4.1]{VerbekenGinisPolygon2026}:
\begin{equation}
 (s-dq)\theta+\sum_{j=1}^d\phi_j=2\pi(r-dp).
 \label{eq:return-argument-sum}
\end{equation}
This equality adds the angles of the factors in the product equation.
The product equation alone determines their sum only modulo $2\pi$.
The return-path construction determines the exact multiple of $2\pi$,
namely $r-dp$. This additional information is needed to apply the bound
on the modulus.
The reduction can keep fixed the contact index $\kappa$ and the positive
product $\prod_i a_i$ supplied by Lemma~\ref{lem:positive-contact}.
\end{lemma}
\begin{proof}
Let $f$ be the characteristic polynomial of $P$. After the relabelling in
Lemma~\ref{lem:positive-contact}, we have $P=C^\kappa B(\gamma)$ with
$a_i>0$ for every $i$. Fix $f$, $\kappa$ and $c=\prod_i a_i>0$.
Consider all pairs $(A,v)$ in which $A=C^\kappa B(\gamma)$ has
characteristic polynomial $f$, satisfies $\prod_i a_i=c$, and has an
eigenvector $v$ for $\lambda$ whose coordinates are cyclically ordered
vertices and satisfy $\max_i|v_i|=1$. This set is nonempty and compact.
Indeed, $a_i\ge c>0$ because $a_i\le1$ and their product is $c$.
Limits retain the eigenvalue and normalisation, and the smallest possible
order $n$ excludes coinciding coordinates or coordinates that cease to
be vertices. Among all pairs with the fewest positive coefficients
$\gamma_i$, choose a pair $(A,v)$ whose polygon
$\conv\{v_0,\ldots,v_{n-1}\}$ has the smallest area.

Suppose $\gamma_i=t>0$ and $\gamma_{i+1}=0$. Define $U_i(t)$ to be
the identity matrix with row $i$ changed so that
$(U_i(t)x)_i=(1-t)x_i+tx_{i-1}$. This matrix is stochastic and
invertible because $t<1$. Replace row $i$ of $B$ by row $i$ of the
identity matrix to obtain $\widehat B$, and put $H=C^\kappa\widehat B$.
Then
\[
 A=HU_i(t),\qquad
 A'=U_i(t)H=U_i(t)A U_i(t)^{-1}.
\]
Thus $A'$ is stochastic and has the same characteristic polynomial $f$.

The calculation in \cite[Lemma~2.4]{VerbekenGinisPolygon2026} gives
\[
 a_i'=1,\qquad a_{i+\kappa}'=a_i a_{i+\kappa},\qquad
 a_j'=a_j\quad(j\notin\{i,i+\kappa\}).
\]
These identities preserve $c$ and the positivity of every $a_j$.
The new eigenvector is $v'=U_i(t)v$: it replaces $v_i$ by a point
strictly between $v_{i-1}$ and $v_i$, leaving all other coordinates
fixed. The smallest possible order and the contact theorem preserve
the vertex order, the contact representation and $\kappa$.

The number of positive $\gamma_i$ cannot increase under this operation,
and our choice prevents it from decreasing. Suppose their indices
formed more than one consecutive run in the cyclic order. Apply the
operation separately to the last index of each run. Each replacement
gives a proper subpolygon. If another vertex has modulus $1$, the
replacement preserves normalisation and decreases the area, contrary
to our choice. Thus every run would have to end at the unique vertex
of modulus $1$, which is impossible for two runs. The positive indices
therefore form one cyclic interval. If it does not contain every index,
repeating the operation moves a positive coefficient from index $i$ to
$i+\kappa$ modulo $n$. If the destination already has a positive
coefficient, the two combine and their number decreases. Such a collision is excluded by minimality.
The sequence of replacements in \cite[Proposition~2.5]{VerbekenGinisPolygon2026}
therefore applies, with $f$, $\kappa$ and $c$ fixed throughout.

The first-return theorem \cite[Theorem~3.5]{VerbekenGinisPolygon2026}
requires boundary contact for every invariant polygon with at most $n$
vertices. This holds because $\lambda$ is a boundary eigenvalue at its
smallest possible matrix order. That theorem argues by constructing
an invariant polygon with a vertex whose image lies strictly inside it.
Note that the auxiliary polygons do not need to preserve $f$ or $c$. The theorem applies to the pair
we selected and gives the paths between consecutive return bases used
in \cite[Theorem~4.1]{VerbekenGinisPolygon2026}.

A return base is a selected state at which one of these paths begins.
If $\gcd(n,\kappa)=1$, choose $d$ bases in their cyclic order. If fewer
than $d$ coefficients $\gamma_i$ are positive, mark additional existing
states on paths whose transitions have probability $1$ as starting
states, and assign the corresponding cycle weight $b_j=0$.
No new states are introduced. At each base, one choice closes a cycle
of length $q$ and weight $b_j$, the product of its transition probabilities; the other continues towards the next base with weight
$1-b_j$. The $q$-cycles have disjoint vertex sets. Following every
continuation gives a cycle of length $s=n$. If $\gcd(n,\kappa)>1$,
the map $i\mapsto i+\kappa\pmod n$ instead has $d$ disjoint cycles of
length $q$, each with one base. The continuation choices join the bases
into a cycle of length $s<n=qd$.
In either case, contracting path segments whose transitions have
probability $1$ leaves the bases, a possible closing loop at each base,
and a directed cycle through all bases. A simple cycle either uses one
closing loop or follows the continuation route through every base.
Thus these are all the cycles of positive weight, and the cycle of
length $s$ shares a vertex with every cycle of length $q$.

Adding the changes of argument along these return paths identifies the
Farey interval between the endpoints of the given arc and gives
\eqref{eq:return-argument-sum}, as in
\cite[Theorem~4.1]{VerbekenGinisPolygon2026}.
\end{proof}

We now have the ingredients needed to prove
Theorem~\ref{thm:boundary-polynomial}. The preceding reduction preserves
the entire characteristic polynomial and expresses it through the
cycle weights $b_1,\ldots,b_d$. We use the boundary assumption to
determine these weights: the bound on the modulus of $\lambda$ can be
attained only when all $b_j$ are equal. Their common value is the arc
parameter $\beta$, and substituting it gives the required
characteristic polynomial.

\begin{proof}[Proof of Theorem~\ref{thm:boundary-polynomial}]
Let $P$ be any stochastic matrix of order $n$ with eigenvalue $\lambda$.
Lemma~\ref{lem:isospectral-return-form} gives a matrix $A$ with the same
characteristic polynomial.
Lemma~\ref{lem:return-determinant} gives formula
\eqref{eq:return-determinant} for $\det(zI-A)$, with $0\le b_j<1$.
The sum of the factor angles satisfies
\eqref{eq:return-argument-sum}, by
Lemma~\ref{lem:isospectral-return-form}. Because $\lambda$ lies on the boundary,
its modulus attains the bound in \cite[Theorem~5.2]{VerbekenGinisPolygon2026}.
Equality forces every $b_j$ to equal the arc parameter $\beta$.
Every cycle weight satisfies $b_j=\beta>0$. Thus the case in the
preceding lemma where additional starting states are marked with
corresponding weight $b_j=0$ is excluded.
Substituting $b_j=\beta=1-\alpha$ gives
\[
 \det(zI-P)=\det(zI-A)
 =z^{n-qd}(z^q-\beta)^d-\alpha^d z^{n-s}.
\]
Thus every stochastic matrix of order $n$ with eigenvalue $\lambda$
has characteristic polynomial \eqref{eq:rigid-boundary-polynomial}.

Now suppose the arc is of Type III and its endpoints are ordered
$p/q<r/n$. We must show that, in the eigenvector ordering of the
states, every transition $i\to i+r\pmod n$ has positive probability.
The arc lies between the angles $2\pi(r-1)/n$ and $2\pi r/n$, so the
two possible transitions from state $i$ go to $i+r-1$ and $i+r$,
with indices modulo $n$.

Complex conjugation reverses the counterclockwise order of the
eigenvector coordinates. We therefore label the conjugated coordinates by
\[
 w_j=\overline{v_{-j}},
\]
with indices modulo $n$, so that they are again ordered counterclockwise.
In these labels, an original transition with displacement $r$ has
displacement $n-r$, while one with displacement $r-1$ has displacement
$n-r+1$. Complex conjugation also reflects the arc across the real axis
and reverses the order of its endpoint fractions. The endpoints of the
conjugate arc are therefore
\[
 \frac{n-r}{n}<\frac{q-p}{q}.
\]
In particular, its endpoint of denominator $n$ is on the left.

Suppose the probabilities of $i\to i+r$ were not all positive.
The contact theorem would then guarantee that every transition
$j\mapsto j+\kappa'$, with $\kappa'=n-r+1$, has positive probability
in the labels just defined. The preceding reduction can keep this
common index step $\kappa'$ fixed.

If $\gcd(n,\kappa')=1$, the reduction gives an arc with right endpoint
$\kappa'/n$, whose denominator is $n$. This disagrees with the
conjugate arc, whose only endpoint of denominator $n$ is on the left.

If $\gcd(n,\kappa')>1$, the reduction gives an arc whose smaller
endpoint denominator divides $n$. The smaller denominator of the
conjugate arc is $q$. Since $q\ge2$, the identity $qr-pn=1$ rules
out $q\mid n$, so these endpoint denominators also disagree.

The conjugate arc has a unique pair of Farey endpoint fractions,
so both cases are impossible. Consequently, in the original ordering,
$P_{i,i+r}>0$ for every $i$, with indices modulo $n$.
\end{proof}

For the remainder of this section, let $\lambda$ lie strictly between the
endpoints of a Karpelevi\v{c} boundary arc, and let $n$ be the smallest
order of a stochastic matrix with eigenvalue $\lambda$. Write the
endpoint denominators as $q<s$, with $q\ge2$, and put
$d=\lfloor n/q\rfloor$. The arc parameter satisfies $0<\beta<1$.
In particular, $q+s>n$ and $n=\max(qd,s)$.

\begin{lemma}
\label{lem:typeii-branches}
Under the standing minimal-order assumptions, suppose $d\ge2$, and
let $t$ be the number of rows of $P$ with two positive entries.
Then
\[
 \begin{array}{ll}
 n=qd:& d\le t\le d+n-s,\\
 n=s>qd:& d\le t\le d+n-qd.
 \end{array}
\]
In the first case, the graph contains $d$ disjoint positive
$q$-cycles, each of weight $\beta$.
\end{lemma}
\begin{proof}
Theorem~\ref{thm:boundary-polynomial} identifies the characteristic
polynomial before either classification is applied.
For Type~II, $1\le n-s<q$, $\gcd(q,n-s)=1$ and $0<\alpha<1$.
The parametrisation in \cite[Theorem~4.1]{VerbekenGinisTypeII2026}
arranges the states into $d$ disjoint directed $q$-cycles, each of
weight $\beta$. At every state, one transition follows its $q$-cycle.
Some states also have a second transition to the next $q$-cycle,
with the last cycle followed by the first. These and only these
states have two positive entries in their rows.

Let $S_j$ be the set of states with this second transition on cycle
$j$, for $1\le j\le d$. The classification provides nonnegative
integers $h_1,\ldots,h_d$, with $\sum_j h_j=n-s$, such that each
$S_j$ is nonempty and is contained in an interval of $h_j+1$
consecutive positions on its cycle. Consequently,
\[
 t=\sum_{j=1}^d |S_j|,\qquad
 d\le t\le\sum_{j=1}^d(h_j+1)=d+n-s.
\]

For Type~III, $1\le n-qd<q$ and $\gcd(q,n)=1$.
After relabelling, the graph has a positive forward cycle through
all $n$ states. A row has two positive entries precisely when its
state has an additional outgoing transition besides the cycle edge;
call these the branching states.
The classification in
\cite[Theorem~1.2 and Lemma~3.2]{VerbekenGinisTypeIII2026}
divides the $t$ branching states into $d$ nonempty groups.
Successive branching states from different groups are separated
by at least $q$ edges along the forward cycle.

List the branching states in the order of that cycle. A gap is the
number of forward edges from one listed state to the next, including
the gap from the last back to the first. The $t$ gaps sum to $n$.
Let $k$ count changes of group between successive listed states.
Since all $d\ge2$ groups occur, $k\ge d$.
The $k$ gaps between different groups have length at least $q$,
and each remaining gap has length at least one. Hence
\[
 n\ge kq+(t-k)\ge dq+t-d.
\]
The nonempty groups also give $t\ge d$, so
$d\le t\le d+n-qd$.
\end{proof}

\begin{theorem}
\label{thm:minimal-arc-exclusion}
Let $n\ge4$, and let $A$ be a Karpelevi\v{c} boundary arc with endpoint
denominators $q<s$, where $q\ge2$. For each $\lambda$ between the
endpoints of $A$, suppose that a stochastic matrix of order $n$ has
eigenvalue $\lambda$ and that none of smaller order does.
If
\[
 1\le m<\left\lceil\frac{\lfloor n/q\rfloor}{2}\right\rceil,
\]
then the only points of $A$ in $\Omega_{n,m}$ are its two endpoints
on the unit circle.
The range for $m$ is nonempty only when $\lfloor n/q\rfloor\ge3$.
\end{theorem}
\begin{proof}
Suppose that $P\in\Q_{n,m}$ has an eigenvalue $\lambda$ on the arc
$A$, distinct from its two endpoints. In the fixed cyclic order,
\[
 P_{ij}=c_{j-i\bmod n}\qquad(i\in I_{n,m})
\]
for one probability vector $c$. Put $d=\lfloor n/q\rfloor$ and
$b=2m<d$; since $m\ge1$, we have $d\ge3$. The classifications may
relabel the states: their row counts still apply, but their cycle
order need not be this fixed cyclic order.

The Dmitriev--Dynkin contact theorem
(Lemma~\ref{lem:positive-contact}) places $\lambda$ times each
eigenvector coordinate on an edge of the polygon formed by the
coordinates. The corresponding row of $P$ expresses that point as a
convex combination of the polygon's vertices, so only the two endpoints
of that edge can have positive coefficients. Hence every row of $P$
has at most two positive entries.
Let $t$ count the rows with two positive entries. If $c$ had only
one positive entry, then $t\le b<d$, contrary to
Lemma~\ref{lem:typeii-branches}. Thus $c$ has exactly two positive
entries, every interior row has two, and $n-b\le t$.

First suppose $n=qd$. Put $z=n-s$, so $1\le z\le q-1$.
For $q\ge3$, the upper bound $t\le d+z$ contradicts
\[
 n-b-(d+z)\ge d(q-2)+1-z\ge(d-1)(q-2)>0.
\]
For $q=2$, the bounds force $n=2d$, $b=d-1$ and $t=d+1$.
The rows with two strictly positive entries are exactly the interior
rows. Every boundary row therefore has one entry equal to $1$ and all
other entries equal to $0$. Thus each boundary state has only one
possible next state.

The classification provides $d$ disjoint directed two-state cycles.
We regard their vertices as pairs of states while retaining the
original cyclic labels used for the Toeplitz rows. For every pair
$\{x,y\}$,
\[
 P_{xy}P_{yx}=\beta,\qquad 0<\beta<1.
\]
A pair cannot contain two boundary states: both transition probabilities
would then equal $1$, contradicting their product $\beta<1$.
Consequently, the $d-1$ boundary states belong to $d-1$ different pairs,
each with one interior partner state. Exactly two interior states
remain, and they form the last pair.

In each mixed pair, the boundary-to-interior transition has probability
$1$, so the reverse transition has probability $\beta$. Thus $\beta$
is one of the two common probabilities defining the interior rows.
In the pair consisting entirely of interior states, neither transition
can have probability $\beta$: the other transition has probability
strictly below $1$, which would make their product smaller than
$\beta$. Both transitions therefore have the other common probability,
$\gamma\ne\beta$, and $\gamma^2=\beta$.

Let $h$ be the common displacement of transitions with probability
$\gamma$: every interior state $i$ sends that probability to
$i+h\pmod{2d}$. In the pair of interior states, these transitions
go in opposite directions between the same two distinct states.
Thus $2h=0\pmod{2d}$ and $h\ne0$, so $h=d$.

Write the interior interval as $\{a,a+1,\ldots,a+d\}$, with indices
modulo $2d$. Its only pair of states separated by $d$ consists of
its endpoints, so these are the two interior states paired with
each other. Removing them leaves
$I'=\{a+1,\ldots,a+d-1\}$.
Each state in $I'$ is paired with a boundary state, and its transition
to that state has probability $\beta$. Let $k$ be the common
displacement of that probability in $c$. These transitions map
$I'$ bijectively onto $B_{n,m}$, so $I'+k=B_{n,m}$.
Shifting $I'$ by $d$ gives $\{a+d+1,\ldots,a+2d-1\}$, exactly
the boundary states. Hence
\[
 I'+k=B_{n,m}=I'+d.
\]
The interval $I'$ starts at $a+1$: this is its unique state whose
preceding cyclic label lies outside $I'$. The boundary interval starts
at $a+d+1$. Translation by $k$ must send the first of these starting
states to the second, so
\[
 a+1+k\equiv a+d+1\pmod{2d}.
\]
Hence $k\equiv d\equiv h\pmod{2d}$. This assigns the two distinct
positive entries of $c$ to the same displacement, giving a contradiction.

Now suppose $n=s=qd+r$, where $0<r<q$.
Lemma~\ref{lem:typeii-branches} gives
$n-b\le t\le d+r$. Thus $b\ge d(q-1)\ge d$, again a contradiction.

The chain construction in Theorem~\ref{thm:chains} realises both
endpoints with $m=1$, and hence for every $m\ge1$.
\end{proof}

\section{Geometry below saturation}
\label{sec:regions}

The preceding results determine when $\Omega_{n,m}=\Theta_n$.
For smaller $m$, we show that the distance from $0$ to the boundary
varies continuously with direction, and give optimal bounds on the
angles at which the region approaches roots of unity.
Theorem~\ref{thm:local-polygon} gives local equality with the circulant
polygon near certain roots of unity whose orders divide $n$, under
its stated bound on $m$.

\begin{proposition}
\label{prop:padding}
If $3\leq k\leq n$ and $m\geq\ceil{(k-1)/4}$, then
\begin{equation}
 \Theta_k\subseteq\Om_{n,m}.
 \label{eq:padding}
\end{equation}
For every $m\geq1$, the unit-circle points of $\Om_{n,m}$ are exactly
the roots of unity of orders at most $n$.
\end{proposition}
\begin{proof}
For a boundary arc of $\Theta_k$ with smaller endpoint denominator
$q\ge2$, put $r=\lfloor k/q\rfloor$ and let $s$ be the other endpoint
denominator. Theorem~\ref{thm:chains} gives a stochastic matrix of order
$n$ whose characteristic polynomial is given by
\[
 z^{n-k}\chi_{k;q,s,\alpha}(z).
\]
Thus each point of the arc occurs as an eigenvalue of one of these
matrices. Because $r\le\lfloor k/2\rfloor\le2m$, the construction
belongs to $\Q_{n,m}$.

For the two Type 0 arcs, use a Markov chain on $n$ states whose last
$k$ states form a directed cycle. Let $P$ be its transition matrix.
Only the last row of $P$ can differ from the cyclic successor
transition $i\mapsto i+1\pmod n$, so $P\in\Q_{n,1}$.
For $0\le t\le1$, form the convex combination
\[
 P_t=(1-t)I+tP.
\]
At each step, this chain stays in its current state with probability
$1-t$ and follows the transition rule of $P$ with probability $t$.
The matrix $P_t$ still belongs to $\Q_{n,1}$ and has eigenvalues
$1-t+t e^{\pm2\pi i/k}$. As $t$ varies from $0$ to $1$, these trace
the two Type 0 segments.
These constructions cover $\partial\Theta_k$. Every point of
$\Theta_k$ lies on a segment joining $1$ to a boundary point, and
\eqref{eq:fill-from-one} puts that whole segment in $\Omega_{n,m}$.

The same chain with a final cycle of length $q$, for $1\le q\le n$,
realises every $q$th root of unity with $m=1$.
Conversely, every eigenvalue of modulus $1$ of a stochastic matrix
comes from a closed communicating class of states. If that class has
period $h$, the eigenvalue is an $h$th root of unity. Since $h$ does
not exceed the number of states in the class, we have $h\le n$.
\end{proof}

At order $2$, $m=0$ already gives every eigenvalue in
$\Theta_2=[-1,1]$. Note that this does not imply
$\Theta_2\subseteq\Omega_{n,0}$ at larger orders: when $n$ is odd,
$-1\notin\Omega_{n,0}=H_n$. Thus a construction that works with
$m=0$ need not retain that value of $m$ when the matrix order increases.
Taking $k=4$ in \eqref{eq:padding} gives
$\Theta_4\subseteq\Omega_{n,m}$ for $n\ge4$ and $m\ge1$.

\begin{theorem}
\label{thm:radial}
Let $n\ge4$ and $0\le m\le\lfloor n/2\rfloor$.
There is a positive continuous semialgebraic function $R_{n,m}$ on the
unit circle such that
\begin{equation}
 \Om_{n,m}=\{ru:|u|=1,\ 0\le r\le R_{n,m}(u)\}.
 \label{eq:radial-region}
\end{equation}
Thus $R_{n,m}(u)$ is the greatest distance from $0$ reached by the region
in direction $u$. The boundary is the simple closed curve traced by
$u\mapsto R_{n,m}(u)u$, and the region is homeomorphic to the closed disc.
The boundary is a finite union of real-analytic arcs and their endpoints.
Writing $z=x+iy$, each arc lies on a curve $F(x,y)=0$ for some
nonzero real polynomial $F$.
\end{theorem}
\begin{proof}
For $m=0$, all claims follow from the polygon description in
\eqref{eq:circulant-polygon}. Assume $m\ge1$, and recall that every
entry of $J$ equals $1/n$.
By \eqref{eq:averaging-spectrum}, each segment from $0$ to a non-Perron
eigenvalue lies in the region; the segment $[0,1]$ already lies in $H_n$.
Compactness gives a maximum radius $R_{n,m}(u)$ in every direction.
By \eqref{eq:padding}, the region contains $\Theta_4$ and hence the
diamond $H_4$, so every radius is positive.

For nonreal $\lambda\in\Om_{n,m}$ and $0<t<1$, the point $t\lambda$
lies in the interior of the triangle in \eqref{eq:fill-from-one}.
Thus every such $t\lambda$ is an interior point of the region.

Compactness gives
$\limsup_{v\to u}R_{n,m}(v)\le R_{n,m}(u)$.
For nonreal $u$, the interior property just proved gives the opposite
inequality for the lower limit. Thus $R_{n,m}$ is continuous away from
$u=\pm1$. At those two directions, $R_{n,m}(\pm1)=1$.
The diamond $H_4$ gives lower bounds on nearby radii tending to $1$,
while the unit disc gives the upper bound $1$. Hence $R_{n,m}$ is
continuous everywhere.
Since it is positive and continuous, the radial map
\[
 ru\longmapsto rR_{n,m}(u)u,\qquad |u|=1,\quad 0\le r\le1,
\]
with $0\mapsto0$, is a homeomorphism from the closed disc onto the
region. Its restriction to the unit circle traces the boundary once.

Finally,
\begin{equation}
 \begin{gathered}
 x+iy\in\Om_{n,m}\quad\Longleftrightarrow\quad
 \exists P\in\Q_{n,m}:\\
 \operatorname{Re}\det((x+iy)I-P)=
 \operatorname{Im}\det((x+iy)I-P)=0.
 \end{gathered}
 \label{eq:polynomial-membership}
\end{equation}
The row constraints and the two determinant equations are polynomial
conditions with rational coefficients. Eliminating the matrix entries
shows that $\Omega_{n,m}$ is semialgebraic over $\mathbb Q$.
The graph of the maximum radius is described by requiring
$ru\in\Omega_{n,m}$ and excluding $tu\in\Omega_{n,m}$ for every $t>r$,
where $|u|=1$ and $r\ge0$. Applying real quantifier elimination again
shows that this graph is semialgebraic over $\mathbb Q$
\cite{BasuPollackRoy2006}. The boundary is
therefore a one-dimensional semialgebraic set. Its decomposition into
real-analytic algebraic arcs is finite, and its being a simple closed
curve makes the remaining points endpoints of those arcs.
\end{proof}

The conclusions do not hold at order $3$ in general. For $m=1$, the family
contains every stochastic matrix of order $3$, and
\[
 \Omega_{3,1}=\Theta_3=H_3\cup[-1,-1/2].
\]
The radial function equals $1$ at direction $-1$ and tends to $1/2$
through nonreal directions, so it is discontinuous there.
Figure~\ref{fig:order-three-radial-jump} shows this exception.

\begin{figure}[htbp]
\centering
\includegraphics[width=.55\textwidth]{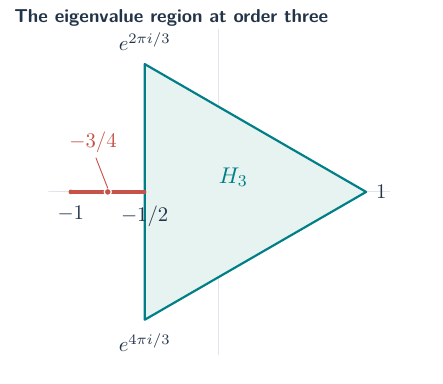}
\caption{At order $3$, $\Omega_{3,1}=\Theta_3=H_3\cup[-1,-1/2]$,
where $H_3$ is the triangle with vertices the cube roots of unity.
The greatest distance from $0$ is $1$ in the negative real direction,
whereas it tends to $1/2$ in nearby nonreal directions.
Removing the marked point $-3/4$ disconnects the region.}
\label{fig:order-three-radial-jump}
\end{figure}

\begin{corollary}
\label{cor:lipschitz}
For $n\ge4$ and $0\le m\le\lfloor n/2\rfloor$, the radial function
$R_{n,m}$ is Lipschitz continuous: there is a constant $L>0$ such that
\[
 |R_{n,m}(u)-R_{n,m}(v)|\le L|u-v|
 \qquad(|u|=|v|=1).
\]
As a function of the angle, $R_{n,m}(e^{i\theta})$ has finite left
and right derivatives at every $\theta$, and is real analytic except
at finitely many angles modulo $2\pi$.
\end{corollary}
\begin{proof}
For $m=0$, the region is the regular polygon $H_n$. Between the
adjacent vertex directions $2\pi j/n$ and $2\pi(j+1)/n$, its maximum
radius is
\[
 R_{n,0}(e^{i\theta})=
 \frac{\cos(\pi/n)}{\cos\!\left(\theta-(2j+1)\pi/n\right)}.
\]
The denominator is at least $\cos(\pi/n)>0$. These finitely many
formulas join continuously and have bounded derivatives, giving
Lipschitz continuity and finite one-sided derivatives at the vertex
directions.
For $m\ge1$ and nonreal $\lambda\in\Om_{n,m}$,
\eqref{eq:fill-from-one} gives
\begin{equation}
 B\left(t\lambda,\frac{(1-t)|\operatorname{Im}\lambda|}{n}\right)
 \subseteq\Om_{n,m}\qquad(1/2\le t<1).
 \label{eq:interior-ball}
\end{equation}
Consider the triangle with vertices $-1/(n-1)$, $1$, and $\lambda$.
The distances from $t\lambda$ to the lines containing the sides
$[\lambda,1]$ and $[\lambda,-1/(n-1)]$ are at least the stated radius.
Its distance to the real side is $t|\operatorname{Im}\lambda|$.

On any closed arc $K$ avoiding $\pm1$, let
\[
 c=n^{-1}\min_{u\in K}|\operatorname{Im}(R_{n,m}(u)u)|>0.
\]
For $u,v\in K$ with $0<\delta=|u-v|<c/4$, take $t=1-2\delta/c$.
The disc in \eqref{eq:interior-ball}, centred at $tR_{n,m}(u)u$,
contains $tR_{n,m}(u)v$. Hence $R_{n,m}(v)\ge tR_{n,m}(u)$.
Interchanging $u,v$ gives $|R_{n,m}(u)-R_{n,m}(v)|\le2|u-v|/c$.

Near any direction $e^{i\phi}$, the rational parametrisation
$u(s)=e^{i\phi}(1-s^2+2is)/(1+s^2)$ makes
$f(s)=R_{n,m}(u(s))$ semialgebraic. It is real analytic on finitely
many intervals, and $f'$ has a one-sided limit, finite or infinite,
at every endpoint \cite{BasuPollackRoy2006}. At $\phi=0,\pi$,
the diamond $H_4$ and the unit disc give $f(s)=1+O(|s|)$.
An infinite limit of $f'$ would contradict this bound by the mean
value theorem. Thus the derivative is bounded near both real
directions as well. Hence every direction has a neighbourhood on
which $R_{n,m}$ satisfies a Lipschitz bound.

These neighbourhoods cover the unit circle: every point belongs to at
least one of them. Compactness lets us choose finitely many that still
cover the whole circle. Let $M$ be the largest of their Lipschitz
constants. For any $u,v$ on the circle, divide a shorter circular arc
joining them into sufficiently short pieces, each lying in one of
these neighbourhoods. Summing the local bounds over those pieces
bounds $|R_{n,m}(u)-R_{n,m}(v)|$ by $M$ times the length of the arc.
That length is at most $(\pi/2)|u-v|$, so
\[
 |R_{n,m}(u)-R_{n,m}(v)|\le\frac{\pi M}{2}|u-v|.
\]
This gives a global Lipschitz constant $L=\pi M/2$.
The same finite interval decomposition then gives
finite one-sided derivatives at every angle and real analyticity
apart from finitely many angles modulo $2\pi$.
\end{proof}

\begin{theorem}
\label{thm:sharp-angles}
Let $m\ge1$ and let $\zeta$ be a primitive $q$th root of unity,
$2\le q\le n$. Define
\[
 N_{n,m}(q)=
 \begin{cases}
 n,&q\mid n,\\
 q\min\{2m,\floor{n/q}\},&q\nmid n.
 \end{cases}
\]
To measure how $\lambda$ approaches $\zeta$, rotate $\zeta$ to $1$
and take the signed angle of $1-\overline\zeta\lambda$ with the positive
real axis. The possible limits of this angle, as
$\lambda\in\Om_{n,m}\setminus\{\zeta\}$ tends to $\zeta$, form
exactly the interval
\begin{equation}
 \left[-\frac\pi2+\frac\pi{N_{n,m}(q)},
        \frac\pi2-\frac\pi{N_{n,m}(q)}\right].
 \label{eq:angle-interval}
\end{equation}
\end{theorem}
\begin{proof}
Choose matrices $P_j\in\Q_{n,m}$ with eigenvalues
$\lambda_j\to\zeta$, and pass to a subsequence on which the matrix
entries converge to those of $P_0$. Let $r$ be the multiplicity of
$\zeta$ as an eigenvalue of $P_0$. By Theorem~\ref{thm:tangent},
any limiting angle has absolute value at most $\pi/2-\pi/(qr)$.
Always $qr\le n$. If $q\nmid n$, the same argument as in the proof of
Proposition~\ref{prop:nondivisor} gives $r\le2m$.
Thus $qr\le N_{n,m}(q)$, which gives the bounds in
\eqref{eq:angle-interval}.

If $q\mid n$, the two sides of $H_n$ meeting $\zeta$ attain the
lower and upper bounds in \eqref{eq:angle-interval}. Their points are eigenvalues of convex combinations
of the circulant permutation matrices corresponding to the endpoints.
Now suppose $q\nmid n$ and set
\[
 r=\min\{2m,\floor{n/q}\}.
\]
Write $\zeta=e^{2\pi ip/q}$ and choose
$s_0\in\{1,\ldots,q-1\}$ with $ps_0\equiv1\pmod q$.
Put $s=(r-1)q+s_0$. Then $r\le s<qr\le n$, so the matrices in
Theorem~\ref{thm:chains} belong to $\Q_{n,m}$ because
$\ceil{r/2}\le m$. The expansion in
\eqref{eq:endpoint-expansion}, applied to their polynomial, gives
an eigenvalue branch $\lambda(\alpha)$, with $\lambda(0)=\zeta$,
satisfying
\[
 \lambda(\alpha)=\zeta\left[1+\frac{u-1}{q}\alpha+O(\alpha^2)\right],
 \qquad u=e^{-2\pi i/(qr)}.
\]
Hence $1-\overline\zeta\lambda(\alpha)$ has limiting angle
$\arg(1-u)=\pi/2-\pi/(qr)$, the upper endpoint angle in
\eqref{eq:angle-interval}.
Choosing $ps_0\equiv-1\pmod q$ and
$u=e^{2\pi i/(qr)}$ gives the lower endpoint angle.

For either endpoint angle, let $P(\alpha)$ be the corresponding matrix
family and write
\[
 1-\overline\zeta\lambda(\alpha)=A\alpha+O(\alpha^2),
 \qquad\operatorname{Re}A>0.
\]
For fixed $c\ge0$ and sufficiently small $\alpha>0$, the matrix
\[
 \widetilde P_c(\alpha)=(1-c\alpha)P(\alpha)+c\alpha J
\]
belongs to the same class and has eigenvalue
$\widetilde\lambda_c(\alpha)=(1-c\alpha)\lambda(\alpha)$.
Therefore
\[
 1-\overline\zeta\widetilde\lambda_c(\alpha)
 =(A+c)\alpha+O(\alpha^2).
\]
As $c$ increases, $\arg(A+c)$ takes every angle strictly between
the chosen endpoint angle and $0$, as well as that endpoint angle at $c=0$.
Finally, $(1-\alpha)P(0)+\alpha J$ has eigenvalue
$(1-\alpha)\zeta$, whose approach angle is $0$.
\end{proof}

\begin{corollary}
\label{cor:root-slopes}
Let $n\ge4$, $m\ge1$, and let $\zeta=e^{i\phi}$ be a primitive
$q$th root of unity, $2\le q\le n$. Set $N=N_{n,m}(q)$. Then
\begin{equation}
 R_{n,m}(e^{i(\phi+h)})
 =1-\tan(\pi/N)|h|+o(|h|)\qquad(h\to0).
 \label{eq:root-slopes}
\end{equation}
This shows that the radius has left derivative
$\tan(\pi/N)$ and right derivative $-\tan(\pi/N)$ at $\phi$.
At $\zeta=1$, the same statements hold with $N=n$ for every $m$.
\end{corollary}
\begin{proof}
For $q\ge2$, the hypotheses give $N>2$. Put
\[
 R_h=R_{n,m}(e^{i(\phi+h)}),\qquad
 \lambda_h=R_he^{i(\phi+h)}.
\]
Then
\[
 1-\overline\zeta\lambda_h
 =1-R_h\cos h-iR_h\sin h.
\]
Since $R_h\to1$, its imaginary part is $-h+o(h)$ and its real
part is $1-R_h+O(h^2)$. As $h\to0$, every limit of the signed angle that
$1-\overline\zeta\lambda_h$ makes with the positive real axis has
absolute value at most $\pi/2-\pi/N$, by
Theorem~\ref{thm:sharp-angles}.
Comparing the real part with the absolute value of the imaginary part
therefore gives
\[
 \liminf_{h\to0}\frac{1-R_h}{|h|}\ge\tan(\pi/N).
\]

To obtain the reverse bound, use the branches attaining the lower
and upper limiting approach angles; if $q\mid n$, use the corresponding sides of $H_n$.
On each branch, choose the angle $\theta(\alpha)$ continuously
with $\theta(0)=\phi$. Its derivative with respect to $\alpha$
at $0$ is nonzero, so the branch reaches every sufficiently close
angle on the corresponding side of $\phi$. Writing
$h=\theta(\alpha)-\phi$, the endpoint expansion gives
\[
 1-|\lambda(\alpha)|=\tan(\pi/N)|h|+o(|h|).
\]
The maximum radius $R_h$ is at least $|\lambda(\alpha)|$.
Consequently,
\[
 \limsup_{h\to0}\frac{1-R_h}{|h|}\le\tan(\pi/N),
\]
which proves the formula for $q\ge2$.

At $\zeta=1$, use the Type~0 segments from $1$ to
$e^{\pm2\pi i/n}$. They belong to every class and form the boundary
of $\Theta_n$ near $1$, so the same calculation applies with $N=n$.
\end{proof}

Matching the lower and upper limiting approach angles does not imply
matching nearby boundary points: for $n=9,m=1$, the angles at a primitive cube root agree with
those of the unrestricted region, but Theorem~\ref{thm:local-polygon}
gives local equality with $H_9$; see Figure~\ref{fig:local-polygon-comparison}.

\section{Remaining questions}
\label{sec:questions}

Determine explicit equations for the boundary of $\Omega_{n,m}$ when
\[
 0<m<\left\lceil\frac{n-1}{4}\right\rceil.
\]
Section~\ref{sec:regions} describes the regularity of its maximum radius
and its limiting approach angles at roots of unity, but does not
identify all its boundary arcs.

For $m<j_n^*$, which lists of $n$ eigenvalues occur together in
$\Q_{n,m}$? Appendix~\ref{app:extensions} treats matrices whose second
row is a cyclic shift of their first, with all other rows free probability
vectors. Near each strictly positive matrix in this family, every
sufficiently small real change in the characteristic coefficients that
retains the eigenvalue $1$ can be obtained within the same family.
A global description of the attainable complete spectra remains open.

Let $A$ be a Karpelevi\v{c} arc on $\partial\Theta_n$ that also lies on
$\partial\Theta_N$ for some $N<n$. Determine $A\cap\Omega_{n,m}$.
To find a matrix $P\in\Q_{n,m}$ with a given eigenvalue $\lambda\in A$,
the transition probabilities to be chosen are the entries $c_k$ of
the common vector defining the interior rows and the entries $P_{bj}$
of the free boundary rows $b\in B_{n,m}$.
Appendix~\ref{app:feasibility} bounds how many entries of these vectors
need to be positive. Once their positive positions are chosen, all
other entries are zero, and only the probabilities at those positions
remain to be determined. Explicit inequalities describing the
intersection and efficient algorithms for deciding membership remain open.

For each even $n\ge4$, construct an explicit stochastic matrix $P$
with strictly positive rational entries whose complete eigenvalue list
no matrix in $\Q_{n,n/2-1}$ can reproduce.
In this family, the second of the two consecutive interior rows is
obtained by shifting the first one position to the right; all other
rows are free probability vectors.
Corollary~\ref{cor:positive-joint-witnesses} proves existence but does
not give concrete entries in each order. Choosing rational entries
gives a finite description by fractions and allows exact computation
of the characteristic polynomial. For the chosen $P$, give an explicit
positive lower bound on the distance from its ordered
characteristic-polynomial coefficient list to
$\operatorname{coeff}(\Q_{n,n/2-1})$, together with a proof of that bound.
Such a positive gap verifies that no matrix $Q\in\Q_{n,n/2-1}$
has the same characteristic polynomial as $P$, and hence none has
the same complete eigenvalue list.

\appendix

\section{Feasibility and support reductions}
\label{app:feasibility}

We give necessary and sufficient conditions for a complex number
$\lambda$ to occur as an eigenvalue of a matrix in $\Q_{n,m}$.
These conditions use the common probability vector defining the
interior rows and the coordinates of an eigenvector. We also bound
the number of strictly positive entries needed in that common vector
and in each free boundary row.

\begin{proposition}
\label{thm:membership}
A number $\lambda\in\mathbb C$ belongs to $\Omega_{n,m}$ if and only if
there are a probability vector $c=(c_0,\ldots,c_{n-1})\in\mathbb R^n$
and a nonzero vector $v=(v_0,\ldots,v_{n-1})\in\mathbb C^n$ such that
\begin{align}
 \lambda v_i&=\sum_{k=0}^{n-1}c_kv_{i+k}
 &&(i\in I_{n,m}),\label{eq:membership-interior}\\
 \lambda v_b&\in\conv\{v_0,\ldots,v_{n-1}\}
 &&(b\in B_{n,m}).\label{eq:membership-boundary}
\end{align}
Here $I_{n,m}=\{m,\ldots,n-m-1\}$ indexes the interior rows, and
$B_{n,m}$ indexes the first and last $m$ rows. Subscripts in
\eqref{eq:membership-interior} are read modulo $n$.
Each boundary row can be chosen with at most three positive entries.
The vector $c$ can simultaneously be chosen with at most
$\min\{n,2|I_{n,m}|+1\}$ positive entries.
If $\lambda$ is real, $v$ can be chosen real; each boundary row then
needs at most two positive entries, and $c$ needs at most
$\min\{n,|I_{n,m}|+1\}$.
\end{proposition}
\begin{proof}
Use the cyclic shifts of $c$ as interior rows and the convex combinations
as boundary rows. These choices give $P\in\Q_{n,m}$ with $Pv=\lambda v$;
conversely, an eigenvector of such a matrix satisfies the displayed
conditions.

For fixed $v$, Carath\'eodory's theorem in $\mathbb R^2$ gives each
boundary row with at most three positive entries. The possible $c$ form
a nonempty compact polytope with at most $2|I_{n,m}|+1$ real equality
constraints, including its sum. A vertex of this polytope has at most
that many positive entries. Choosing such a vertex leaves the boundary
conditions unchanged, so the bounds hold simultaneously.
For real $\lambda$, choose $v$ real and apply the same argument on
$\mathbb R$: two entries suffice in each boundary row, and the system
for $c$ has at most $|I_{n,m}|+1$ equations.
\end{proof}

\begin{proposition}
\label{prop:boundary-two-support}
Every nonreal $\lambda\in\partial\Omega_{n,m}$ is an eigenvalue of
a matrix $P\in\Q_{n,m}$ in which each of the $2m$ boundary rows has
at most two positive entries. The probability vector $c$ defining
the cyclic shifts in the interior rows can simultaneously be chosen
with at most $\min\{n,2|I_{n,m}|+1\}$ positive entries.
Given any stochastic matrix $A\in\Q_{n,m}$ with eigenvalue $\lambda$,
one can replace each boundary row by a probability vector with at most
two positive entries while preserving all interior rows and the
eigenvalue $\lambda$. The eigenvalue $\lambda$ may be repeated,
and the given matrix may be reducible.
\end{proposition}
\begin{proof}
When $m=0$, there are no boundary rows and the bound on $c$ is $n$,
so no change is needed. Assume $m\ge1$.
Let $\lambda\in\partial\Omega_{n,m}$ be nonreal, and let
$P\in\Q_{n,m}$ be a stochastic matrix with eigenvalue $\lambda$.
Choose a boundary row $b\in B_{n,m}$ and write its entries as
$p_j=P_{bj}$, so $p_j\ge0$ and $\sum_jp_j=1$.

For each column $j$, let $P^{(j)}$ be the matrix obtained from $P$
by replacing row $b$ by the row with entry $1$ in column $j$ and
$0$ in every other column. Put
\[
 f_j(z)=\det(zI-P^{(j)}).
\]
The $b$th row of $zI-P$ is the convex combination, with weights
$p_j$, of the $b$th rows of $zI-P^{(j)}$; all other rows agree.
Linearity of the determinant in that row therefore gives
\[
 \det(zI-P)=\sum_jp_jf_j(z),\qquad
 0=\sum_jp_jf_j(\lambda).
\]

Consider the complex numbers $f_j(\lambda)$ for the columns where
$p_j>0$. The second equality expresses $0$ as a convex combination
of these numbers, with every participating weight positive.
If these numbers did not lie on one line in the complex plane,
$0$ would lie in the interior of their convex hull.
Each polynomial $f_j(z)$ varies continuously with $z$, so $0$
would remain in that hull for every $z$ sufficiently close to
$\lambda$. For each such $z$, we could choose nonnegative row
entries $q_j$, summing to $1$, with
\[
 \sum_jq_jf_j(z)=0.
\]
Replacing row $b$ by these entries would give a stochastic matrix
in $\Q_{n,m}$ with eigenvalue $z$. Thus a whole neighbourhood of
$\lambda$ would lie in $\Omega_{n,m}$, contrary to
$\lambda\in\partial\Omega_{n,m}$.

The complex numbers $f_j(\lambda)$ with $p_j>0$ therefore lie on
one line. Their convex hull is a line segment, or a single point,
containing $0$. Express $0$ as a convex combination of at most
two endpoints of this hull, and put the corresponding weights
in the corresponding columns of row $b$, with all other entries
zero. The determinant at $\lambda$ is still zero, so $\lambda$
remains an eigenvalue. Row $b$ now has at most two positive entries.

Apply the same argument to each remaining boundary row of the
current matrix. All interior rows stay fixed throughout these changes.
Finally, choose an eigenvector $v$ of the resulting matrix for
$\lambda$ and apply the choice of $c$ in
Proposition~\ref{thm:membership}. This gives the stated bound on
the number of positive entries of $c$. The boundary rows already
constructed still satisfy their eigenvector equations with this
same $v$, so they can be retained when the interior rows are
changed to use the new $c$.
\end{proof}

Fix a number $\lambda\notin H_n$ whose membership in $\Omega_{n,m}$
we want to determine. Following the criterion of
Proposition~\ref{thm:membership}, we must find a probability vector
$c$ and a nonzero eigenvector $v$. The interior equations say that
$(\lambda I-C(c))v$ is zero at every coordinate indexed by $I_{n,m}$.
We can therefore write
\[
 (\lambda I-C(c))v=U_Bh,
\]
where $h\in\mathbb C^{|B_{n,m}|}$ contains the coordinates indexed
by the boundary states. The map $U_B$ places those coordinates in
an $n$-component vector and puts zero at the interior-state coordinates.

Every eigenvalue of $C(c)$ lies in $H_n$, so $\lambda I-C(c)$ is
invertible. The criterion becomes the existence of $c$ and $h$
satisfying
\begin{equation}
 \begin{gathered}
 v=(\lambda I-C(c))^{-1}U_Bh,\qquad \|h\|_2=1,\\
 \lambda v_b\in\conv\{v_0,\ldots,v_{n-1}\}\quad(b\in B_{n,m}).
 \end{gathered}
 \label{eq:boundary-reduction}
\end{equation}
Since $\lambda I-C(c)$ is invertible and $v\ne0$, the equation
\[
 (\lambda I-C(c))v=U_Bh
\]
implies $h\ne0$. Dividing both $v$ and $h$ by $\|h\|_2$ therefore
gives the normalisation $\|h\|_2=1$.
Thus the $2m$ complex coordinates to be found are those of $h$,
in place of the $n$ coordinates of $v$. The real probabilities
$c_k$ still have to be chosen. Once $c$ and $h$ are chosen,
the displayed formula determines $v$, and the convex-hull
conditions determine whether suitable boundary-row probabilities
exist.

For arcs inherited from smaller orders, the next proposition also bounds
the number of positive entries in the shared probability vector $c$.

\begin{proposition}
\label{prop:shared-support}
Let $\lambda\in\Theta_n$ be nonreal and suppose
$t\lambda\notin\Theta_n$ for every $t>1$. Let $N$ be the smallest order of a stochastic
matrix with eigenvalue $\lambda$. If $P\in\Q_{n,m}$ has eigenvalue
$\lambda$, choose a probability vector defining its interior rows
and let $p$ be its number of positive entries. Then
\[
 (p-2)(N-2m)\le p(n-N).
\]
In particular, let $\lambda$ lie on a Karpelevi\v{c} boundary arc,
distinct from its two endpoints. Write the endpoint denominators as
$q<s$, where $q\ge2$, and put $d=\lfloor n/q\rfloor$.
If $1\le m<\lceil d/2\rceil$ and $P\in\Q_{n,m}$ has eigenvalue
$\lambda$, its shared vector has at most two positive entries.
There is also a matrix in $\Q_{n,m}$ with eigenvalue $\lambda$
and at most two positive entries in every row.
\end{proposition}

\begin{proof}
Put $b=2m$ and choose a communicating class $S$ of $P$ whose
submatrix $A$ has eigenvalue $\lambda$. It is irreducible and
substochastic, with Perron root $0<\rho\le1$.
If $\rho<1$, a positive Perron vector $u$ and $D=\diag(u)$
make $\rho^{-1}D^{-1}AD$ stochastic with eigenvalue $\lambda/\rho$,
contradicting the assumed maximality on its ray. Hence $\rho=1$.
A positive left Perron vector then forces $A\one=\one$, so $S$
is closed.

Choose $v\ne0$ with $Av=\lambda v$ and put
$K=\conv\{v_i:i\in S\}$. If $K$ has $h$ vertices, expressing
$\lambda$ times each vertex as a convex combination of vertices
gives a stochastic matrix of order $h$ with eigenvalue $\lambda$,
so $h\ge N$.
Every $\lambda v_i$ lies on $\partial K$. Otherwise row $i$ could
be replaced by strictly positive probabilities while preserving
$Av=\lambda v$. For the resulting irreducible stochastic matrix
$A'$, let $\pi>0$ be stationary. Since $\lambda\ne1$,
$\pi^{\mathsf T}v=0$. Let $\mathbf f_i\in\mathbb R^S$ have coordinate
$1$ at state $i$ and coordinate $0$ at every other state. Then
\[
 A_\varepsilon=A'-\varepsilon\mathbf f_i\pi^{\mathsf T}
 \quad\hbox{satisfies}\quad A_\varepsilon v=\lambda v.
\]
This operation subtracts $\varepsilon\pi^{\mathsf T}$ from row $i$.
For small $\varepsilon>0$, this matrix is nonnegative and irreducible
with a deficient row sum, hence has Perron root below $1$.
Perron normalisation would again give an eigenvalue farther out
on the same ray.

Choose one state for each vertex of $K$, forming $V\subseteq S$,
and put $M=|S\cap I_{n,m}|$ and $e=n-h$.
Each shared row in $S$ can use at most two of these vertex
representatives, because its image lies on $\partial K$.
Each positive shift in the shared vector misses $V$ at only $e$
of the $n$ source labels. Counting incidences therefore gives
\[
 p(M-e)\le2M.
\]
For $p\ge2$, the bounds $M\ge N-b$ and $e\le n-N$ prove
$(p-2)(N-b)\le p(n-N)$. For $p=1$, this is just $b\le n$.

For the stated arc, $N=\max\{qd,s\}$, $r=n-N<q$ and $b<d$,
with $d\ge3$. Thus
\[
 N-b-3r\ge(d-3)(q-1)+1>0.
\]
If $p\ge3$, the general bound would give
$N-b\le p r/(p-2)\le3r$, a contradiction.
Finally, $\lambda\in\partial\Theta_n\cap\Om_{n,m}$ lies on
$\partial\Om_{n,m}$. Apply Proposition~\ref{prop:boundary-two-support}
while keeping the shared vector fixed.
\end{proof}

Fix $n,m$ and a number $\lambda$ satisfying the arc hypotheses of
Proposition~\ref{prop:shared-support}. We want to determine whether
$\lambda\in\Omega_{n,m}$, meaning that some stochastic matrix
$P\in\Q_{n,m}$ has eigenvalue $\lambda$. As in the criterion of
Proposition~\ref{thm:membership}, we use a common probability vector
$c$ and free boundary rows. Under the present hypotheses, each
can be restricted to at most two positive entries.

First choose two possible positions in $c$ and two possible columns
in each boundary row. Put $t_0$ and $1-t_0$ in the chosen positions
of $c$. For each $b\in B_{n,m}$, put $t_b$ and $1-t_b$ in the chosen
columns of row $b$. All other entries are zero, and the interior
rows are the prescribed cyclic shifts of $c$.
For each choice of positions, these probabilities determine a
candidate matrix $P\in\Q_{n,m}$ through at most $2m+1$ real
parameters: $t_0$ and the $2m$ numbers $t_b$, all in $[0,1]$.

We must find values of these parameters together with a vector
$v=(v_0,\ldots,v_{n-1})\in\mathbb C^n$, $v\ne0$, satisfying
$Pv=\lambda v$. The coordinates of $v$ are additional unknowns;
the count $2m+1$ concerns only the transition probabilities.
A solution for at least one choice of positions gives a matrix
in the family with the prescribed eigenvalue $\lambda$.

\begin{corollary}
\label{cor:effective-regions}
Fix $n,m$. If $\lambda$ has algebraic real and imaginary parts, we
can decide whether $\lambda\in\Omega_{n,m}$ by real quantifier
elimination. This procedure replaces the existence of unknown real
matrix entries by equivalent polynomial conditions on the real
and imaginary parts of $\lambda$.

If $n\ge4$ and $A$ is a fixed Karpelevi\v{c} boundary arc of
$\Theta_n$, then $A\cap\Omega_{n,m}$ is a finite union of closed
subarcs and isolated points. The endpoints of these subarcs and the
isolated points have algebraic real and imaginary parts and can
also be determined by real quantifier elimination.
\end{corollary}
\begin{proof}
By definition, $\lambda=x+iy$ belongs to $\Omega_{n,m}$ precisely
when some stochastic matrix $P$ with the required Toeplitz interior
satisfies $\det((x+iy)I-P)=0$. The restrictions on the entries of $P$,
together with the real and imaginary parts of this determinant
equation, form a finite system of polynomial equations and
inequalities with rational coefficients. Applying real quantifier
elimination to this system gives equivalent polynomial conditions
on $x,y$ alone \cite{BasuPollackRoy2006}. For algebraic $x,y$, these
conditions can be checked exactly.

For $n\ge4$, Theorem~\ref{thm:radial} describes
$\partial\Theta_n$ as a continuous radial curve semialgebraic
over $\mathbb Q$. The endpoints of $A$ are roots of unity,
so their real and imaginary parts are algebraic. The closed
angular sector between the rays from $0$ through these endpoints
can therefore be described by polynomial inequalities with
algebraic coefficients. Its intersection with $\partial\Theta_n$
is $A$. Consequently, $A\cap\Omega_{n,m}$ is a compact
semialgebraic subset of an arc, hence a finite union of closed
subarcs and isolated points.

The component endpoints and isolated points form a finite set
defined by polynomial conditions with algebraic coefficients.
Their coordinates are therefore algebraic and can be determined
by the same elimination procedure.
\end{proof}

\section{Local freedom of characteristic-polynomial coefficients}
\label{app:extensions}

Let $P$ be a real $n\times n$ matrix, $n\ge3$, whose rows sum to $1$.
Row $1$ is obtained by shifting row $0$ cyclically one position to the right:
\[
 P_{1,j+1}=P_{0j},
\]
with column indices read modulo $n$.
Each of rows $2,\ldots,n-1$ can be chosen independently of the other
rows, subject only to its entries summing to $1$. Denote this family
of real matrices by $\mathcal L$. At this stage, entries may be
negative. Requiring all entries to be nonnegative selects the
stochastic matrices.
For even $n$, this is the two-interior-row pattern of
$\Q_{n,n/2-1}$ after a cyclic relabelling.

Since $P\one=\one$, vectors differing by a constant vector have images
differing by that same constant vector. We therefore identify $x$ and
$y$ whenever $y=x+c\one$ for some real number $c$. Write $[x]$ for
the resulting class and define
\[
 V=\mathbb R^n/\operatorname{span}\{\one\}.
\]
If $y=x+c\one$, then $Py=Px+c\one$, so $[Py]=[Px]$.
Consequently,
\[
 \overline P:V\longrightarrow V,\qquad [x]\longmapsto[Px]
\]
is a well-defined linear map. This is the map induced by $P$ on $V$.

Each class has a unique representative whose coordinates sum to zero,
obtained by subtracting the coordinate average. Using these
representatives, applying $\overline P$ means multiplying by $P$ and
then subtracting the coordinate average of the result.
Write $\operatorname{End}(V)$ for the space of all linear maps from
$V$ to itself.
Passing from $P$ to $\overline P$ separates off one forced copy of
the eigenvalue $1$.
We prove below that every linear map on $V$ comes from exactly one
matrix in $\mathcal L$. We use this correspondence to vary
characteristic-polynomial coefficients and to determine which
stochastic matrices can be approximated by strictly positive matrices
with the same eigenvalues. Note that repeated eigenvalues are allowed.

Given a linear map $M:V\to V$, we seek a real matrix $A$ acting on
$\mathbb R^n$ whose rows sum to $1$ and whose action on classes agrees
with $M$:
\[
 A\one=\one,\qquad [Ax]=M([x])
 \quad\text{for every }x\in\mathbb R^n.
\]
We call such a matrix a \emph{row-sum-one lift of $M$}.
Such a lift always exists. Identify $V$ with the mean-zero
representatives described above, let $A$ act on this subspace as $M$,
and let $A$ fix $\one$.

The matrix $A$ just constructed has row sums $1$ and induces $M$ on
$V$, but its first two rows need not satisfy the cyclic shift relation.
We now adjust $A$ to enforce that relation while preserving its action
on $V$. Set $d_j=A_{0j}-A_{1,j+1}$ and solve
\[
 w_{j+1}-w_j=d_j,\qquad \sum_jw_j=0.
\]
The two rows of $A$ each sum to $1$, so $\sum_jd_j=0$.
Starting with any $w_0$, the recurrence determines the other
coordinates and returns to $w_n=w_0$. Subtracting the mean of
these coordinates makes their sum zero without changing the
differences. Any two solutions of the difference equations differ
by a constant vector, so this mean-zero solution is unique.

Set $P=A+\one w^{\mathsf T}$. Adding $w^{\mathsf T}$ to every row
preserves row sums, and the difference equations give
$P_{1,j+1}=P_{0j}$. The added vector $\one(w^{\mathsf T}x)$ is
constant for every $x$, so $P$ still induces $M$ on $V$.
Thus $P\in\mathcal L$ is a lift of $M$.

If two matrices in $\mathcal L$ induce the same map on $V$, their
difference has image in the constant vectors. It therefore equals
$\one u^{\mathsf T}$ with $\sum_ju_j=0$. The row relation forces
$u_{j+1}=u_j$, so $u$ is constant and hence zero.
The map
\[
 \mathcal L\longrightarrow\operatorname{End}(V),\qquad
 P\longmapsto\overline P
\]
is bijective and affine. Denote its inverse by $L$. Thus $L(M)$ is the
unique real matrix satisfying the row conditions and inducing $M$
on $V$.

Writing $\chi_A(z)=\det(zI-A)$ for a square matrix $A$, choose a
basis of $\mathbb R^n$ beginning with $\one$. In this basis,
$P=L(M)$ is block triangular, with diagonal blocks $1$ and a matrix
representing $M$.
Hence
\[
 \chi_{L(M)}(z)=(z-1)\chi_M(z).
\]
Let $\mathbf f_i\in\mathbb R^n$ have coordinate $1$ at position $i$
and coordinate $0$ at every other position. A row vector defines a
functional on $V$ exactly when its entries sum to zero.
Let $\ell[x]=x_0-x_1$, represented by
$\mathbf f_0^{\mathsf T}-\mathbf f_1^{\mathsf T}$, and put $r=\ell M$.
For $P=L(M)$, $r$ is represented by the difference
of rows $0$ and $1$ of $P$. If $p$ is the first row, the right-shift
relation gives
\[
 r_j=p_j-p_{j-1},\qquad \sum_jp_j=1.
\]
These equations determine the first row $p$. Indeed, successive
differences give
\begin{equation}
 p_j=p_0+\sum_{k=1}^j r_k\qquad(1\le j<n).
 \label{eq:first-row-entries}
\end{equation}
Summing these expressions and using $\sum_jp_j=1$ gives
\begin{equation}
 p_0=\frac1n-\frac1n\sum_{k=1}^{n-1}(n-k)r_k.
 \label{eq:first-row-diagonal}
\end{equation}

\begin{lemma}
\label{lem:coefficient-openness}
For each $d$, the map sending a real $d\times d$ matrix to the
ordered coefficients of its monic characteristic polynomial is open.
Thus every neighbourhood of a matrix $M$ contains matrices whose
characteristic polynomials cover a neighbourhood of $\chi_M$
among real monic polynomials of degree $d$.
\end{lemma}
\begin{proof}
If $u$ is cyclic for a $d\times d$ matrix $M$, then
$u,Mu,\ldots,M^{d-1}u$ is a basis and
\[
 \chi_{M+uw}(z)=\chi_M(z)-w\operatorname{adj}(zI-M)u
\]
for every row vector $w$.
Consider the linear map $w\mapsto w\operatorname{adj}(zI-M)u$
into polynomials of degree at most $d-1$.
The \emph{resolvent} of $M$ is the inverse $(zI-M)^{-1}$, defined
when $z$ is not an eigenvalue. For sufficiently large $|z|$,
\[
 (zI-M)^{-1}=\sum_{k=0}^{\infty}\frac{M^k}{z^{k+1}}.
\]
If $w\operatorname{adj}(zI-M)u$ is the zero polynomial, dividing
by $\chi_M(z)$ and using this expansion gives $wM^ku=0$ for
$0\le k<d$. The displayed basis then gives $w=0$.
Both the space of row vectors and the polynomial space have
dimension $d$, so this linear map is invertible. All nearby monic
polynomials therefore occur near $M$.

For arbitrary $M$, joining Jordan blocks belonging to the same
eigenvalue by arbitrarily small nonzero entries gives cyclic
isospectral matrices nearby. Use real $2\times2$ blocks for
conjugate eigenvalues. Every neighbourhood of $M$ thus contains
a cyclic matrix with polynomial $\chi_M$ and realises a
neighbourhood of that polynomial.
\end{proof}

\begin{proposition}
\label{prop:two-row-submersion}
\label{prop:free-row-positive}
Let $n\ge3$ and let $P\in\mathcal L$ be stochastic, with every
entry in rows $2,\ldots,n-1$ strictly positive. If its first row differs from $\mathbf f_0^{\mathsf T}$, it is a limit of strictly positive
matrices in $\mathcal L$ with the same characteristic polynomial,
and every neighbourhood of $P$ contains stochastic matrices in
$\mathcal L$ realising all sufficiently nearby real monic degree-$n$
polynomials that vanish at $1$.
If that row is $\mathbf f_0^{\mathsf T}$, its polynomial lies on the
boundary of the attainable coefficient set. Under the stated
positivity assumption, this case occurs exactly when $1$ is repeated.
\end{proposition}
\begin{proof}
Put $M=\overline P$. If $r=\ell M$ and $\ell$ are independent,
choose $a\in V$ with $\ell a=0$ and $ra=1$. For small $t>0$,
$S=I-ta r$ is invertible and $\ell S=\ell$. Hence
\[
 \ell(S^{-1}MS)=rS=(1-t)r.
\]
The matrix $L(S^{-1}MS)$ has first row
$(1-t)p+t\one^{\mathsf T}/n$, so both constrained rows are positive.
The other rows remain positive by continuity.

If $r=\lambda\ell$, then $r_0=\lambda$, $r_1=-\lambda$, and
$r_j=0$ for $j\ge2$. Substituting into
\eqref{eq:first-row-diagonal} and \eqref{eq:first-row-entries} gives
\[
 p_0=\frac{1+(n-1)\lambda}{n},\qquad
 p_j=\frac{1-\lambda}{n}\quad(j\ne0).
\]
For $-1/(n-1)<\lambda<1$, the matrix is already positive.
The endpoint $\lambda=1$ makes both constrained states absorbing.
At $\lambda_*=-1/(n-1)$, only $P_{00}$ and $P_{11}$ are zero.

We now track $p_0$ under changes of $M$ by writing it as $1/n$
plus the value of the row $\ell M$ on a fixed vector $g$.
Let $g\in V$ be represented by
$(0,-(n-1)/n,\ldots,-1/n)^{\mathsf T}$. The displayed formula for
$p_0$ in \eqref{eq:first-row-diagonal} gives $p_0=1/n+(\ell M)g$
and $\ell g=(n-1)/n$.
If $(\lambda_*I-M)g\ne0$, choose $H$ so that
$\ell H(\lambda_*I-M)g>0$. Along
$M(t)=(I+tH)^{-1}M(I+tH)$, the row-$0$ diagonal entry of
$L(M(t))$ satisfies
\[
 p_0'(0)=\ell H(\lambda_*I-M)g>0.
\]
Both zero entries become positive, and all other entries stay positive.

If $Mg=\lambda_*g$, choose $a\in\ker\ell$ with
$(M-\lambda_*I)a\ne0$. Such an $a$ exists unless $M=\lambda_*I$,
since $V=\operatorname{span}\{g\}\oplus\ker\ell$.
For $T_s=I+sa\ell/(\ell g)$ and $N=T_s^{-1}MT_s$,
\[
 \ell N=\lambda_*\ell,\qquad
 (N-\lambda_*I)g=s(M-\lambda_*I)a\ne0.
\]
The matrix $L(N)$ has the same rows $0$ and $1$ as $P$, so the
preceding argument applies arbitrarily close to $P$.
The scalar exception cannot occur:
$L(\lambda_*I)=\lambda_*I+(1-\lambda_*)J$ has zero diagonal
entries in rows $2,\ldots,n-1$.

Every neighbourhood of $P$ now contains a strictly positive
stochastic matrix in $\mathcal L$ with the same characteristic
polynomial. Around such a matrix, choose a sufficiently small
neighbourhood in $\mathcal L$, contained in the original neighbourhood
of $P$, so that all entries remain positive. The affine correspondence
$M\mapsto L(M)$ carries neighbourhoods in $\operatorname{End}(V)$ to
neighbourhoods in $\mathcal L$.
Lemma~\ref{lem:coefficient-openness} allows all sufficiently small
changes in $\chi_M$. Since
\[
 \chi_{L(M)}(z)=(z-1)\chi_M(z),
\]
the corresponding stochastic matrices realise all sufficiently nearby
real monic degree-$n$ polynomials that vanish at $1$.

Finally, a proper closed class must lie in $\{0,1\}$, since every
entry in rows $2,\ldots,n-1$ is positive. The cyclic row relation permits such a class
only for $p=\mathbf f_0^{\mathsf T}$. Otherwise $P$ is irreducible and $1$
is simple. In the absorbing case, $1$ is repeated; moving one copy
to $1+\varepsilon$ gives arbitrarily close unattainable polynomials.
Since the coefficient image of the compact stochastic part of
$\mathcal L$ is closed, $\chi_P$ lies on its boundary.
\end{proof}

\section*{Funding}

Vincent Ginis acknowledges support from the Research Foundation -- Flanders
(FWO) under grants No.~G032822N and G0K9322N.

\section*{Use of AI tools}
The authors wanted to gain more understanding about how eigenvalue
regions change. Discussions with AI tools led to the idea of using
a Toeplitz interior. The main mathematical ideas originated with
the authors and developed through one author's presentation and
other talks and discussions at the ILAS 2026 meeting.
The manuscript
was developed through collaboration between the authors and AI tools.
Discussions with GPT-6 Astra helped develop arguments and close gaps in
the proofs. OpenAI Codex, using GPT-6.1 Sol, assisted with rewriting the
exposition, preparing the figures, and writing and running code for
numerical checks. The authors take full responsibility for the
mathematical claims, proofs, computations, figures, and final text.

\end{document}